\documentclass[12pt,letter]{amsart}
\usepackage[
top=2.7cm, bottom=2.7cm, inner=2.7cm, outer=2.7cm,
marginparwidth=2cm %necessary to avoid breaking fixme notes.
]{geometry}

\usepackage[utf8]{inputenc}
\usepackage[english]{babel}
\usepackage{comment}

\usepackage{pgfplots}
\pgfplotsset{compat=1.15}
\usepackage{mathrsfs}
\usetikzlibrary{arrows}

\usepackage{amsmath,amsthm,amssymb,latexsym,mathtools,enumerate,caption,subcaption,hyperref}
\usepackage{epigraph,graphicx}

\usepackage{todonotes}
\usepackage{bbm}
\presetkeys%
    {todonotes}%
    {inline,backgroundcolor=yellow}{}

\graphicspath{ {./}{./figures/} }
\DeclareGraphicsExtensions{.png,.pdf}

\theoremstyle{plain}
\newtheorem{theorem}{Theorem}[section]
\newtheorem*{theorem*}{Theorem}
\newtheorem*{conj*}{Conjecture}
\newtheorem{lemma}[theorem]{Lemma}
\newtheorem{prop}[theorem]{Proposition}
\newtheorem{cor}[theorem]{Corollary}
\newtheorem{thmx}{Theorem}
\theoremstyle{definition}
\newtheorem{definition}[theorem]{Definition}
\newtheorem{rem}[theorem]{Remark}

\newtheorem{ex}[theorem]{Example}

\theoremstyle{remark}

\numberwithin{equation}{section}
\numberwithin{theorem}{section}
\numberwithin{table}{section}
\numberwithin{figure}{section}

\newcommand{\bC}{\mathbb{C}}
\newcommand{\setRS}{\widehat{\mathbb{C}}}
\newcommand{\bR}{\mathbb{R}}
\newcommand{\bN}{\mathbb{N}}

\newcommand{\deltacrit}{\delta_{\operatorname{crit}}}
\newcommand{\sing}{{\operatorname{Sing}_F}}

\renewcommand{\:}{\colon}
\newcommand{\att}{\operatorname{att}}

\renewcommand{\leq}{\leqslant}

\renewcommand{\geq}{\geqslant}

\numberwithin{equation}{section}

\title
{Conformal measures of (anti)holomorphic correspondences}
\author[N.~Hemmingsson]{Nils Hemmingsson}
\address{Nils~Hemmingsson, Department of Mathematics, Stockholm University, SE-106 91 Stockholm,
      Sweden}
\email{nils.hemmingsson@math.su.se}
\author[X.~Li]{Xiaoran Li}\address{Xiaoran~Li, School of Mathematical Sciences, Peking University, Beijing 100871, China.}
\email{2000010758@alumni.pku.edu.cn}
\author[Z.~Li]{Zhiqiang Li}\address{Zhiqiang~Li, School of Mathematical Sciences \& Beijing International Center for Mathematical Research, Peking University, Beijing 100871, China.}
\email{zli@math.pku.edu.cn}

\begin{document}
%\today{}

\begin{abstract}
    In this paper, we study the existence and other properties of conformal measures on limit sets of (anti)holomorphic
    correspondences. We show that if the critical exponent satisfies $1\leq \deltacrit(x) <+\infty,$ the correspondence $F$ is relatively hyperbolic on the limit set $\Lambda_+(x)$, and $\Lambda_+(x)$ is minimal, then $\Lambda_+(x)$ admits a non-atomic conformal measure for $F$ and the Hausdorff dimension of $\Lambda_+(x)$ is strictly less than $2$. As a special case, this shows that for a parameter $a$ in the interior of a hyperbolic component of the modular Mandelbrot set, the limit set of the Bullett--Penrose correspondence $F_a$ has a non-atomic conformal measure and its Hausdorff dimension is strictly less than $2$. The same results hold for the LLMM correspondences, under some extra assumptions on their defining function $f$.
\end{abstract}
\maketitle
\tableofcontents
\section{Introduction}\label{sec:introduction}

A holomorphic correspondence $F\:z\to w$ is a multivalued map on the Riemann sphere $\setRS$ defined by a polynomial relation in two variables $P(z,w)=0$. In the same way, an antiholomorphic correspondence is defined by the polynomial relation $P(\overline{z},w)=0$. Whenever we write ``(anti)holomorphic correspondence'', we refer either a holomorphic or antiholomorphic correspondence. (Anti)holomorphic correspondences may be seen as generalizations of rational maps and finitely generated Kleinian groups as follows. If $R=\frac{p(z)}{q(z)}$ is a rational map, then the correspondence defined by the polynomial relation
\begin{equation*}
	P(z,w)=p(z)-wq(z)=0
\end{equation*}
is readily seen to represent the map $R$. Similarly, if $G$ is a finitely generated Kleinian group, generated by the M\"obius transformations
\begin{equation*}
	\gamma_i=\frac{a_iz+b_i}{c_iz+d_i},\qquad  i\in \{1,\, \dots,\, n\},
\end{equation*}
then the $n$-to-$n$ correspondence defined by 
\begin{equation*}
	P(z,w)=\prod_{i=1}^n(a_iz+b_i-w(c_iz+d_i))=0
\end{equation*}
has the same full orbits as that of the group $G$ acting on $\setRS$. 

In 1922, Fatou \cite{fatou1922iteration}  initiated the study of holomorphic correspondences and noted that there are many similarities but also dissimilarities between iterations of general holomorphic correspondences and those of rational functions or finitely generated Kleinian groups. Nonetheless, in \cite{Fatou1929}, he put forth the idea that the two latter areas could be studied through the more general lens of holomorphic correspondences. 

The general theory of holomorphic correspondences is still quite limited, but some very interesting examples are known. 
For instance, in \cite{Bullett1994MatingQM}, the following family of 2-to-2 correspondences with a remarkable property was introduced.
Let $F_a\:z\to w$ be the correspondence defined by
\begin{equation*}
	\biggl(\frac{aw-1}{w-1}\biggr)^2+\frac{(aw-1)}{(w-1)}\frac{(az+1)}{(z+1)}+\biggl(\frac{az+1}{z+1}\biggr)^2=3.
\end{equation*}
It is shown that certain members of this family are \emph{matings} between a quadratic polynomial and $\operatorname{PSL}(2,\mathbb Z)$.
A correspondence $F\:\setRS \to \setRS$ is, in colloquial terms, a mating between a rational map $R$ and a group $G$ if there exists a partition of $\setRS$ into disjoint and under $F$ completely invariant sets $A$ and $B$, such that the following holds. On the set $A$, the correspondence is equivalent to the action of $G$ and on certain subsets $B_i$ of $B$ and suitable forward or backward branches $f_i$ of $F$, $B_i$ is $f_i$-invariant and $f_i$ is conjugate to $R$ on $B_i$. 
The precise definition of matings may vary and depend on the conjugation, but the ones most closely related to this paper may be found in \cite{bullett2017mating} or \cite{lyubich2023antiholomorphic}. In \cite{Bullett1994MatingQM} they further conjectured that for all parameters $a$ in the \emph{modular Mandelbrot set}, the correspondence $F_a$ is a mating between a quadratic polynomial with connected Julia set and $\operatorname{PSL}(2,\mathbb Z)$ and that the modular Mandelbrot set is homeomorphic to the ``standard'' Mandelbrot set. In \cite{bullett2017mating}, the former conjecture was modified to account for the existence of parabolic fixed points and was subsequently proved. 
They show that for every parameter in the modular Mandelbrot set, the correspondence is a quasiconformal mating between $\operatorname{PSL}(2,\mathbb Z)$ and a degree two rational map with a parabolic fixed point. In the sequel \cite{bullett2020mating}, they further show that the modular Mandelbrot set is indeed homeomorphic to the Mandelbrot set. They (and we in Theorem~\ref{thm:BP} below) denote by $\Lambda_{a,+}$ the maximal set on which there exists a backward branch of $F_a$, under which $\Lambda_{a,+}$ is invariant, that is hybrid equivalent on $\Lambda_{a,+}$ to the rational map that $F_a$ mates with $\operatorname{PSL}(2,\mathbb Z)$ (see \cite{Douady1985OnTD} for a definition of hybrid equivalence).

In \cite{Lee_2021}, a large family of correspondences with interesting properties were introduced in the antiholomorphic setting. In \cite{lyubich2023antiholomorphic}, it is shown that for every parabolic antirational map with connected Julia set $R$ of degree $d$, there exists an antiholomorphic correspondence in the aforementioned family that is a mating between $R$ and $\mathbb Z/2\mathbb Z*\mathbb Z/(d+1)\mathbb Z$. Recently, in \cite{bullett2024matingparabolicrationalmaps}, it was shown that any degree $d$ rational map that has a parabolic fixed point of multiplier 1, and a completely invariant and simply connected immediate basin of attraction is mateable with the Hecke group $H_{d+1}.$

The main subject of this paper is \emph{conformal measures} on limit sets of correspondences. In \cite{patterson}, Patterson introduced and constructed conformal measures on limit sets of Fuchsian groups, and in \cite{sullivan}, Sullivan demonstrated that the same ideas may be used to to construct conformal measures on Julia sets of rational maps. These are measures that under elements of the group or under branches of the inverse of the rational map (away from critical values) transform according to \eqref{eq:conf} that is given in Definition~\ref{def:conformalmeasures} below. As such, they have proved to be a powerful tool to study the Hausdorff measure of the limit or Julia sets. For instance, in \cite{urbanskidenker1} and \cite{urbanskidenker6}, it is proved that the Julia set of a parabolic rational map has Hausdorff dimension strictly less than $2$ and is equal to the unique exponent $\delta$ appearing in Definition~\ref{def:conformalmeasures} for which the measure is non-atomic. Before presenting our definition of conformal measures on limit sets of correspondences, we need the following notion.

\begin{definition}
    Let $F\:\setRS\to \setRS$ be an (anti)holomorphic correspondence and $x\in \setRS$. The set $\bigcap _{n=0}^{+\infty} \overline{\bigcup_{k=n}^{+\infty} F^k(x)}$ is the \emph{forward limit set with respect to $x$}  and is denoted by $\Lambda_+(x)$.
\end{definition}

    Throughout the text, $x$ will often be fixed, and in this case we will omit ``with respect to $x$'', and simply say \emph{forward limit set}.
In this paper, $|Df|$ denotes the absolute value in the spherical metric of the total derivative of $f$. We say that $(A,f)$, where $A\subseteq \setRS$ is a connected Borel set and $f$ a branch of $F$, is a \emph{special pair} of a subset $\Lambda\subseteq \setRS$ 
 when the branch $f$ of $F$ is defined and injective on $A$, and $f(A\smallsetminus \Lambda)\cap \Lambda= \emptyset$. We are ready to state the main definition of the present text.

\begin{definition}\label{def:conformalmeasures}
    A Borel probability measure $\mu$ with support contained in a subset $\Lambda\subseteq \setRS$ is \emph{$\delta$-conformal} for $F$ and $\Lambda$ if for each special pair $(A,f)$ of $\Lambda$,
    \begin{equation}\label{eq:conf}
        \mu(f(A))=\int_A \!|Df|^\delta \,\mathrm{d}\mu.
    \end{equation}
    A measure is \emph{conformal} for $F$ and $\Lambda$ if it is $\delta$-conformal for $F$ and $\Lambda$, for some real number $\delta$.
\end{definition}
When $F$ and $\Lambda$ are clear from the situation, we will simply say that the measure $\mu$ is $\delta$-conformal, and that in this situation, that $\mu$ is conformal.
If there exists a $\delta$-conformal measure for $F$ and $\Lambda$, we shall also say that $\Lambda$ admits a $\delta$-conformal measure for $F$.

The condition that $f(A\smallsetminus \Lambda)\cap \Lambda= \emptyset$ is included because, in general, the limit sets on which we will construct conformal measures will only be forward, and not backward, invariant under $F$. Throughout the text,  $\operatorname{HD}(S)$ denotes the Hausdorff dimension of the set $S$ (see e.g., \cite{Beardon2000-ns}). 

Two of our main results are the existence of non-atomic conformal measures for two particular well-studied families.

The first family, mentioned above, was introduced in \cite{Lee_2021} and studied thoroughly in e.g., \cite{lyubich2023antiholomorphic}. Let $d\geq 1$. For each rational map $f \colon \widehat{\mathbb{C}} \rightarrow \widehat{\mathbb{C}}$ that is univalent on $\overline{\mathbb{D}}$ and of degree $d+1$,  they define the antiholomorphic correspondence 
\begin{equation*}
    F(z) \coloneqq  \biggl\{ w \in \widehat{\mathbb{C}} : \frac{f(w) -f(\eta (z))}{w -\eta (z)} =0 \biggr\}.
\end{equation*}
We call the correspondence $F$ defined by $f$ as above, \emph{the LLMM correspondence defined by $f$} (LLMM is short for Lee, Lyubich, Makarov, and Mukherjee).  

In Section~\ref{sec:misha}, we study a subset of this family of correspondences, where the map $f$ needs to satisfy some extra assumptions, introduced in Definition~\ref{def:finM}. The first main result of this paper is Theorem~\ref{thm:misha} stated below. To keep this section shorter and concise, we refer the reader to Section~\ref{sec:misha} for the required definitions.

\begin{thmx}\label{thm:misha}
    Let $f\in \mathcal{M}$ and $F$ be the LLMM correspondence defined by $f$. Suppose that the map $R$, appearing in Definition~\ref{def:finM}, has an attracting periodic orbit in $\setRS\smallsetminus \mathcal{B}(R)$. Then there exists a non-atomic $\delta$-conformal measure for $F$ and $\Lambda_+$, for some $1\leq \delta<2$ and $1\leq \operatorname{HD}(\Lambda_+)\leq \delta$.
\end{thmx}

The second family is the family of Bullett--Penrose correspondences and we show the following.
    
\begin{thmx}\label{thm:BP}
    Let $F_a$ be the Bullett--Penrose correspondence for a parameter $a$
    in the interior of a hyperbolic component of the modular Mandelbrot set. Then there exists a non-atomic $\delta$-conformal measure for $F_a$ and $\partial\Lambda_{a,+}$ for some $1\leq \delta<2$ and $1\leq \operatorname{HD}(\partial \Lambda_{a,+})\leq \delta$.
\end{thmx}

Theorems~\ref{thm:misha} and \ref{thm:BP} are in fact applications of the more general Theorems~\ref{thm:main} and \ref{thm:existencemmeasure}. In order to introduce these results, we require a few more definitions.

The construction of conformal measures that Patterson \cite{patterson} and Sullivan \cite{sullivan} carried out, that we here emulate, employs the Poincar\'e series, which is defined in what follows.  The degrees in the variable $z$ (resp.\ $\overline{z}$) of $P(z,w)$ (resp.\ $P(\overline{z},w$))  will be denoted by $d_z$ and the degree in $w $ will be denoted by $d_w$. 
Suppose that $x\in \setRS$ is such that there exists a neighborhood $U$ of $x$ on which all branches of $F^n$ for all $n\geq 1$ are defined. For $n\geq 0$, we denote by $M_n\coloneqq M_n(U)$ the number of branches of $F^n$ defined in $U$, and denote these branches by $f_{n,j}$, where $j=1,\, \dots, \, M_n$. %For each index $I=(i_1,i_2,...,i_n)$, with $1\leq i_j\leq d_w$, we denote by $f_I$ one of the $d_w^n$ branches of $F^n$, so that $$\{f_I:I=(i_1,i_2,\dots,i_n), 1\leq i_j\leq d_w\}$$ contains all branches of $F^n$. Let $|I|$ denote the length of the index $I$. If $|I|=0$, we define $f_I$ to be the identity map (restricted to $U$). 
For $s> 0$, we define the Poincar\'e series
\begin{equation}\label{eq:poincare}
    P_s(x)\coloneqq  \sum_{n=0}^{+\infty} \sum_{j=1}^{M_n}|Df_{n,j}(x)|^s.
\end{equation} 
%where the $'$ of the second summation means that we sum over the \emph{unique} branches $f_I$ of $F^n$ with index $I$ of length $n\geq 0$. 
If there exists no neighborhood of $x$ on which all branches of $F^n$ are defined, or $P_s(x)$ diverges for all $s>0$, we set $P_s(x)=+\infty$ for all $s>0$. Note that if $P_{s}(x)$ converges for some $s>0$, then $P_t(x)$ converges for all $t>s$. We can now define the following important quantity. 
\begin{definition}\label{def:crit}
    If there exist $s>0$ and $t>0$ such that $P_{s}(x)$ converges  and $P_t(x)$ diverges, then 
\begin{equation*}
	\delta_{\operatorname{crit}}(x) \coloneqq \inf\{s>0:P_s(x)<+\infty\}.
\end{equation*}
If no $s>0$ such that $P_s(x)<+\infty$ exists, then $\delta_{\operatorname{crit}}(x)\coloneqq +\infty $, or if $P_s(x)<+\infty$ for each $s>0$, then $\deltacrit(x)\coloneqq 0$. The extended real number $\deltacrit(x)$ is called the \emph{critical exponent of $F$ at $x$}. 
\end{definition}
\begin{rem}
    When $F$ and $x$ are clear from the situation, we will simply call $\deltacrit(x)$ the \emph{critical exponent}.
\end{rem}

The point $\omega\in \setRS$ is a \emph{parabolic} periodic point of $F$ if there exists an integer $q\geq 1$ and a branch $f_I$ of $F^q$ defined in a neighborhood of $\omega$ such that $f_I(\omega)=\omega$ and $Df_I(\omega)$ is a root of unity or, equivalently, there exists an integer $n\geq 1$ such that $Df_I^n(\omega)=1$. We denote by $\Omega_+(x)$ the set of all parabolic periodic points of $F$ in $\Lambda_+(x)$. We shall impose conditions on $F$ that imply that for each $\omega$, there exists a unique branch $T_\omega$ of $F^q$ fixing $\omega$, and that iterates of this branch does not equal the identity, see Section~\ref{sec:importantdefs}. Then there exist integers $n\geq 1$ and $p(\omega)\geq 1$, a complex number $a\neq 0$, and a neighborhood $V_{\omega}$ of $\omega$ where $T_\omega^m$ is defined for $m=1,\,\dots,\,n$ such that for each $z\in V_\omega$,
\begin{equation*}
	T_{\omega}^n(z)=\omega + ({z}-\omega)+a( z-\omega)^{p(\omega)+1}+\cdots
\end{equation*}
or
\begin{equation*}
	T_\omega^n(z)=\omega + (\overline{z}-\overline\omega)+a(\overline z-\overline\omega)^{p(\omega)+1}+\cdots .
\end{equation*} 
This defines $p(\omega)\geq 1$ used in the formulation of Theorem~\ref{thm:main} below.

We now formulate the next main result of this paper. To keep this section more digestible, we defer the rather long definitions of the \emph{singular points} of $F$, \emph{relatively hyperbolic} correspondences, and \emph{minimal} limit sets  to Section~\ref{sec:importantdefs},  see Definitions~\ref{def:sing}, \ref{def:wellbehaved}, and \ref{def:minimal}.

\begin{thmx}\label{thm:main}
    Let $x\in \setRS$ and $F$ be an (anti)holomorphic correspondence that is relatively hyperbolic on $\Lambda_+(x)$ and such that $\Lambda_+(x)$ is minimal. If $\delta_{\operatorname{crit}}(x)$ satisfies \begin{equation*}
    	\sup_{\omega \in \Omega_+(x)}p(\omega)/(p(\omega)+1)<\delta_{\operatorname{crit}}(x)<+\infty,
    \end{equation*}
    then there exists a non-atomic $\deltacrit(x)$-conformal measure for $F$ and $\Lambda_+(x)$, and
    \begin{equation*}
    	\operatorname{HD}(\Lambda_+(x)) \leq \deltacrit(x)<2.
    \end{equation*}
\end{thmx}
\begin{rem}
    By definition, hyperbolicity implies relative hyperbolicity, see Definition~\ref{def:wellbehaved}.
\end{rem}
A correspondence $F$ is \emph{invariantly inverse-like} on a set $S$ if $F(S)\subseteq S$ and for each $w\in F(S)$, there exists a unique $z\in S$ such that $z\in F^{-1}(w)$. If $F\:z\to w$ is invariantly inverse-like on $S$, we define the map $g_{F,S}\:F(S)\to S$ by $g_{F,S}(w)\coloneqq z$. Note that $F=g_{F,S}^{-1}$ on $S$.
The final result of this paper is the following theorem. 
\begin{thmx}\label{thm:existencemmeasure}
        Let $F$ be an invariantly inverse-like (anti)holomorphic correspondence on a closed set $S\subsetneq \setRS$. Suppose that $g_{F,S}$ has an attracting periodic orbit in the interior of $S$, with immediate basin of attraction $\mathcal A$. %Further, suppose that for a.e.\ every $x$ in a neighborhood of the attracting periodic orbit, $\Lambda_+(x)$ is constant (i.e., the same for a.e.\ $x$) and minimal. 
        Suppose further that there exists $x\in \mathcal{A}\smallsetminus \mathcal{PC}_{F^{-1}}$ such that $F$ is relatively hyperbolic on $\Lambda_+(x)$, that $\Lambda_+(x)$ is minimal, and that $\Lambda_+(x)\cap \sing=\emptyset$. Then $\operatorname{HD}(\Lambda_+(x))<2$ and there exists a non-atomic $\delta$-conformal measure for some $1\leq \delta < 2$.
\end{thmx}
The definition of the \emph{postcritical set} $\mathcal{PC}_{F^{-1}}$ of $F^{-1}$ is given in Definition~\ref{def:postcritical}, and the definitions of the \emph{singular points} of $F$, \emph{relatively hyperbolic} correspondences, and \emph{minimal} limit sets are given in Definitions~\ref{def:sing}, \ref{def:wellbehaved}, and \ref{def:minimal}, respectively.

In \cite[Theorem~11, Section~4.3]{MarianneFreiberger2007}, a similar statement to Theorem~\ref{thm:BP} was announced. Unfortunately, though, to the best of the authors of the present paper's knowledge, its proof is incomplete and the methods of \cite{mcmullen} used in \cite{MarianneFreiberger2007} are not directly applicable. To be more precise, in the proof of \cite[Theorem~7, Section~4.2]{MarianneFreiberger2007}, which is used for the proof of \cite[Theorem~11, Section~4.3]{MarianneFreiberger2007}, it is not shown that the conformal measures constructed are not point masses on the parabolic fixed point, which (at least a priori) could be the case because the parabolic fixed point 0 is a critical point of the branch of $F_a$ that does not fix 0. One way to rule out this possibility is the lower bound on $\deltacrit(x)$ appearing in Theorem~\ref{thm:main}.
Theorem~\ref{thm:existencemmeasure} provides sufficient conditions that imply the stronger property $1\leq \deltacrit(x)< 2$. In Section~\ref{sec:bullett}, we demonstrate that if $a$ belongs to the interior of a hyperbolic component of the modular Mandelbrot set, then the Bullett--Penrose correspondences satisfy these conditions for a large set of $x\in \setRS$.

A difference between the settings of Patterson or Sullivan and ours, hinted at above, is that the inverse branches of $R$ or elements of a Kleinian group $G$ cannot have critical points. The possible existence of critical points of holomorphic correspondences is significant, and in certain situations, allow for the existence of discrete conformal measures.
One of the main novel contributions of this paper is that we show how one can handle the existence of critical points of holomorphic correspondences on the correspondences' forward limit sets. In doing so, one of the main new techniques of this paper is developed, culminating in the proof of Lemma~\ref{le:delta}. Here, we study the Poincar\'e series on the boundary of a topological disk that has a repelling periodic point of $F$ in its interior. Studying the Poincar\'e series of these points simultaneously, using a topological argument and the K\"obe distortion theorem gives the desired bound $1\leq \deltacrit(x)$. Another contribution of this paper is the introduction of a new family of correspondences, namely those that are \emph{relatively hyperbolic} on the limit set $\Lambda_+(x)$, and the definition of a \emph{minimal} limit sets, see Section~\ref{sec:importantdefs}. These definitions are intricate and technical and allow us to carefully study conformal measures of (anti)holomorphic correspondences.

The structure of this paper is as follows. In Section~\ref{sec:importantdefs}, we provide the precise definitions and present the setting for our study. In Section~\ref{sec:ex}, we show that, under the assumptions of Theorems~\ref{thm:misha} and~\ref{thm:BP} there is a large set of $x\in \setRS$ such that the correspondence in question and $\Lambda_+(x)$ satisfy the assumptions of Theorem~\ref{thm:existencemmeasure}. This allows us to conclude Theorems~\ref{thm:misha} and~\ref{thm:BP}.  In Section~\ref{sec:2}, using the methods of Patterson and Sullivan, we construct conformal measures on limit sets disjoint from the set of singular points (see Definition~\ref{def:sing}) and on which $F$ is invariantly inverse-like and non-branched (see Definition~\ref{def:non-branched}). To circumvent the consequences of discrete conformal measures, in Section~\ref{sec:3}, we study \emph{open} conformal measures, i.e., measures which are positive on open sets (of $\Lambda_+(x))$. Combining results in Sections~\ref{sec:2} and \ref{sec:3} gives Theorem~\ref{thm:main}. Lastly, the critical exponent $\deltacrit(x)$ is directly studied in Section~\ref{sec:suffcond}, and we find conditions that imply that non-atomic conformal measures exist, and in particular conclude Theorem~\ref{thm:existencemmeasure}. 

\medskip

\noindent\textbf{Acknowledgements.}
We are deeply indebted to M.~Yu.~Lyubich for suggesting the research questions and his interest and support, and to S.~Mukherjee for our discussion and readily answering questions. We also want to thank Yifan~Ying for his careful reading of our manuscript and numerous helpful comments. Z.~Li was partially supported by NSFC Nos.~12471083, 12101017, 12090010, and 12090015.

\section{Singular points, minimal forward limit sets, and relative hyperbolicity}\label{sec:importantdefs}
In this section, we provide the precise setting for the present paper. 

Denote the set of positive integers by $\bN \coloneqq \{1,\,2,\,3,\, \dots \}$.

Let $P(z,w)$ (or $P(\overline z,w)$) be a polynomial with coefficients in $\bC$. The equation $P(z,w)=0$ defines an algebraic curve $\Gamma\subseteq \setRS\times \setRS$. The algebraic curve $\Gamma$, and hence also $P$, defines an \emph{(anti)holomorphic correspondence} $F\:\setRS \to \setRS$ by 
\begin{equation*}
	F(z)\coloneqq\{w: (z,w)\in \Gamma\}.
\end{equation*}
 We shall write $F\colon z\to w$ to indicate the direction of $F$. For $n\geq 2$, we iteratively define
\begin{equation*}
	F^n(z)\coloneqq \bigl\{w\in F(y): y\in F^{n-1}(z)\bigr\}.
\end{equation*}
 
For a correspondence $F\colon\setRS \to \setRS$, the set $S\subseteq \setRS$ is forward invariant if $F(S)\subseteq S$. 
Moreover, recall that $F$ is \emph{invariantly inverse-like} on a set $S$ if $F(S)\subseteq S$ and for each $w\in F(S)$, there exists a unique $z\in S$ such that $z\in F^{-1}(w)$.

For a given correspondence $F\colon z\to w$, defined by a polynomial $P(z,w)$ (or $P(\overline{z},w)$), $F^{-1}\:w\to z$ is the natural inverse correspondence, i.e., the one defined by the same polynomial. Further, $F^0$ is the identity map.

For a neighborhood $U\subseteq \setRS$, we say that all branches of $F$ are \emph{defined} in $U$ if there exist an integer $M$ and single-valued holomorphic functions  $f_j\colon U\to \setRS$ for $j=1,\,\dots,\, M$ such that 
\begin{equation*}
	F(z)=\{f_j(z):j \in \{1, \, \dots, \,  M\}\}
\end{equation*}
for each $z\in U$. In this case, 
\begin{equation*}
	\{ f_{j}:j \in \{1, \, \dots, \,  M\}\}
\end{equation*}
is the set of \emph{branches} of $F$.

\begin{comment}
    To provide the necessary definitions for this paper, we will need the decomposition of $F$ into its irreducible components. To that end, $\Gamma\subseteq \setRS\times \setRS$ can be decomposed into its irreducible components $\Gamma_j$, written formally as 

\begin{equation}\label{eq:Gammacomponents}
    \Gamma=\sum_{i=1}^N m_j\Gamma_j,
\end{equation}
where $m_j\geq 1$, and the $\Gamma_j$ are distinct for distinct $j$.
We denote by $F_j$ the correspondence defined by $\Gamma_j$ and the corresponding polynomial $P_j$. If $F$ is the correspondence defined by the algebraic curve 
\begin{equation*}
	\Gamma=\sum_{i=1}^N m_j\Gamma_j,
\end{equation*}
we define $F_{\operatorname{min}}$ by the algebraic curve
\begin{equation*}\Gamma_{\min}\coloneqq\sum_{i=1}^N \Gamma_j,\end{equation*}
i.e., taking the multiplicities of all of its irreducible components equal to one. Also $P_{\min}(z,w)$ is its corresponding polynomial, i.e., if each the irreducible factors of $P(z,w)$ are denoted by $p_j(z,w)$ and have order $m_j$, we set
\begin{equation*}P_{\min}(z,w)\coloneqq\frac{P(z,w)}{\prod_{j=1}^{N} p_j(z,w)^{m_j-1}}. \end{equation*}

\end{comment}

Let $F$ be a holomorphic correspondence defined by the algebraic curve $\Gamma$. The maps $\pi_z\colon\Gamma\to \setRS$ and $\pi_w\colon \Gamma\to \setRS$ are the projections onto the first and second coordinates, respectively.
There are several definitions of the important notion of \emph{critical values} of a correspondence $F$. Following \cite{Dinh2020DynamicsOH}, the set of \emph{ramification points} of $F$, denoted by $R$, is the finite set that consists of all points $a\in \Gamma$ satisfying that for each neighborhood $W$ of $a$, $\pi_w$ is not injective on some irreducible component of $\Gamma\cap W$. The set of \emph{critical values} $\mathcal{CV}_F$ of $F$ is the set $\pi_w(R)$. Note that for each $n\geq 0$, all branches of $F^n$ are always defined in any simply connected domain not containing points of $\bigcup_{i=0}^{n-1}F^{-i}(\mathcal{CV}_{F^{-1}})$. We can now give the following definition.\begin{definition}\label{def:postcritical}
    The set $\mathcal{PC}_{F^{-1}}\coloneqq \overline{\bigcup_{i=0}^{+\infty}F^{-i}(\mathcal{CV}_{F^{-1}})}$ is called the \emph{postcritical set} of $F^{-1}$.
\end{definition}

We will also need the following definition.
    \begin{definition}
        A point $z\in \setRS$ is a \emph{critical point} (of $F$) if there exists an (anti)holomorphic branch $f$ of $F$ defined in a neighborhood of $z$ such that $Df(z)=0$. The set of critical points of $F$ is denoted by $\mathcal C_F$. 
\end{definition}

\begin{definition}\label{def:non-branched}
If $\Lambda_+(x)\cap \mathcal{CV}_{F^{-1}}=\emptyset$, we say that
    $F$ is \emph{non-branched} on $\Lambda_+(x)$.
\end{definition}
%We shall impose the condition that $x\in \setRS$ is such that $F$ is non-branched on $\Lambda_+(x)$, so that all branches of $F$ are defined in a neighborhood of every point $z\in \Lambda_+(x)$. 
We now provide the definition of \emph{singular points} of $F$, needed for Theorem~\ref{thm:main}.
\begin{definition}\label{def:sing}
    Denote by $S_F$ the finite set that consists of all points $a\in \Gamma$ satisfying that for each neighborhood $W$ of $a$, $\pi_z$ is not injective on $\Gamma\cap W$. The set of \emph{singular points} of $F$, denoted by $\sing$, is the set $\pi_z(S_F)$.
\end{definition} 
\begin{rem}
    Note that $\mathcal{CV}_{F^{-1}}\subseteq \sing$.
\end{rem}
\begin{rem}\label{remark:invlike} 
    If $F$ is invariantly inverse-like on a set $S$ and $x\in S$, then $\Lambda_+(x)\subseteq \overline S$ and if $S$ is closed, then $F$ is invariantly inverse-like on $\Lambda_+(x)$.
\end{rem}

The following definitions will be important. They are in line with the analogous definitions for holomorphic functions, see e.g., \cite{Beardon2000-ns}. 
\begin{definition}\label{def:parabolic}
 A point $z\in \setRS$ is an \emph{attracting} (resp.\ \emph{repelling}) periodic point (of $F$) if there exists an integer $q\geq 1$ and a branch $f_I$ of $F^q$ fixing $z$ such that $|Df_I(z)|<1$ (resp.\ $|Df_I(z)|>1$). Similarly, a point $z\in \setRS$ is a \emph{parabolic} periodic point of $F$ if there exists an integer $q\geq 1$ and a branch $f_I$ of $F^q$ defined in a neighborhood of $z$ such that $f_I(z)=z$ and $Df_I(z)$ is a root of unity or, equivalently, there exists an integer $n\geq 1$ such that $Df_I^n(z)=1$. If there exists an integer $q\geq 1$ such that $z\in F^q(z)$, the minimal such integer is the \emph{period} of $z$. We denote by $\Omega_+(x)$ the set of all parabolic periodic points of $F$ in $\Lambda_+(x)$. 
 \end{definition}

Next, we introduce a technical assumption on $F$ and $\Lambda_+(x)$.

\begin{definition}\label{def:Omega-attracting}
We say that $F$ is \emph{locally $\Omega_+(x)$-attracting} on $\Lambda_+(x)$ if for each $\omega\in \Omega_+(x)$, and each branch $T_{\omega,j}$ of $F^q$ fixing $\omega$ for some $q\geq 1$, there exists a pinched neighborhood $U$ of $\omega$ with the following properties:
\begin{enumerate}
    \smallskip \item\label{it:locally1} For all sufficiently small neighborhoods $V$ of $\omega$, $\Lambda_+(x)\cap V=\Lambda_+(x)\cap V\cap U$. 
    \smallskip \item\label{it:locally2} $T^n_{\omega,j}$ is defined on $U$ for each $n\geq 0$ and for each $z\in U$, $T^n_{\omega,j}(z)\to \omega$ as $n\to +\infty$. 
    %\item For each $\omega\in \Omega_+(x)$ and each branch $T_{\omega,j}$ fixing $\omega$ there exists a pinched neighborhood $U$ of $\omega$, such that for all sufficiently small neighborhoods $V$ of $\omega$, $\Lambda_+(x)\cap V=\Lambda_+(x)\cap V\cap U$. \item $T^n_{\omega,j}$ are defined on $\overline U$ for all $n$ and $T^n_{\omega,j}(z)\to \omega$ for all $z\in  U$.
\end{enumerate}
\end{definition}
Here, by a \emph{pinched neighborhood} of $\omega$ we mean the closure of an open set consisting of finitely many connected components such that the closure of each connected component contains $\omega$. Note that if $\Lambda_+(x)$ is locally $\Omega_+(x)$-attracting, then no iterates of $T_{\omega,j}$ equal the identity map on the set $U$.
\begin{ex}\label{ex:example}
Let $F(z)\coloneqq R^{-1}(z)$, where $R\coloneqq z^2+1/4$. Then the point $\omega=1/2$ is a parabolic fixed point of $F$ and the set $U$ in Definition~\ref{def:Omega-attracting} may be taken as the interior of a circle sector, see Figure~\ref{fig:Uneighborhood}.   
    \begin{figure}
        \centering
        \includegraphics[scale=0.6]{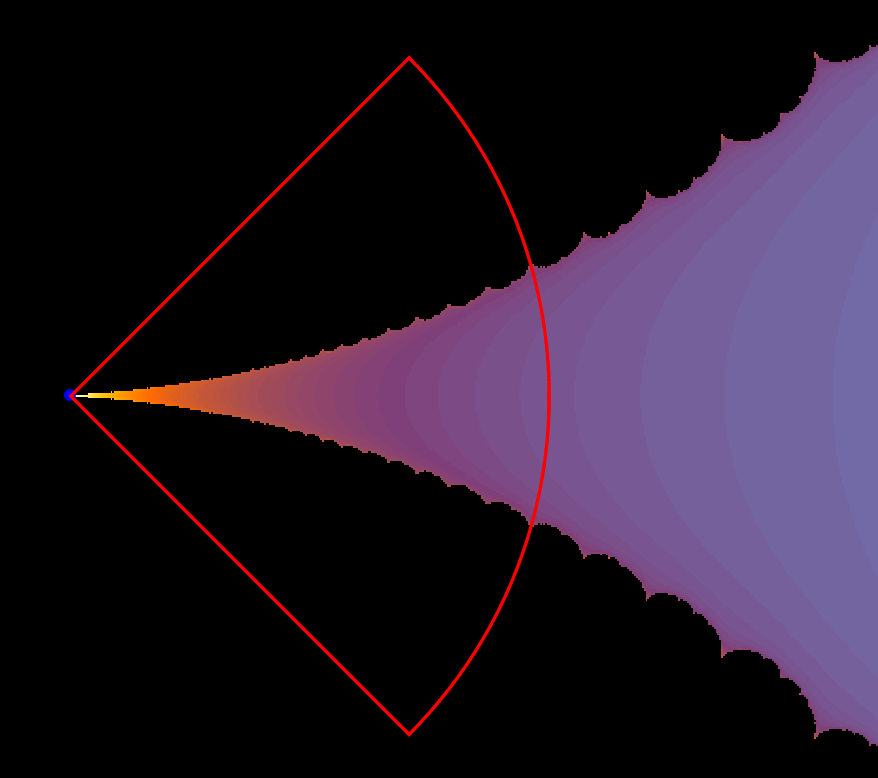}
        \caption{Zoomed-in picture of $\Lambda_+(x)$ near the parabolic fixed point $\omega=1/2$, indicated by a blue dot, for $F(z)\coloneqq R^{-1}(z)$ where $R\coloneqq z^2+1/4$. As long as $x\neq \infty$, $\Lambda_+(x)$ equals the Julia set of $R$, which is the common boundary of the black and the colored regions.  The set $U$ in Definition~\ref{def:Omega-attracting} can be taken as the interior of the red circle sector, provided that it is chosen sufficiently small.}
        \label{fig:Uneighborhood}
    \end{figure}
\end{ex}

\begin{definition}\label{def:wellbehaved}   
Suppose that $F$ is invariantly inverse-like on $\Lambda_+(x)$. We say that $F$ is \emph{relatively hyperbolic} on $\Lambda_+(x)$ if $F$ satisfies the following assumptions:    
\begin{enumerate}
    \smallskip \item\label{it:rel1}  $\Lambda_+(x)\cap \mathcal{CV}_{F^{-1}}=\emptyset$.
    \smallskip \item\label{it:attracting}  $F$ is  locally $\Omega_+(x)$-attracting on $\Lambda_+(x)$.
    \smallskip \item\label{it:thirdrel} $\Lambda_+(x)\cap \mathcal{PC}_{F^{-1}}\subseteq \Omega_+(x)$. 
 \end{enumerate}
  If $\Omega_+(x)=\emptyset$, $\Lambda_+(x)\cap \mathcal{CV}_{F^{-1}}=\emptyset$, and $\Lambda_+(x)\cap \mathcal{PC}_{F^{-1}}= \emptyset$, we say that $F$ is \emph{hyperbolic} on $\Lambda_+(x)$. 
 \end{definition}

 \begin{rem}
Note that if $F$ is hyperbolic on $\Lambda_+(x)$, then $F$ is relatively hyperbolic on $\Lambda_+(x)$. 
 \end{rem}
 \begin{ex}
          Let $R\:\setRS\to \setRS$ be a rational map and $F=R^{-1}$. Take $x\in \setRS$ such that it is not an attracting periodic point of $R$ and $x$ does not belong to a Siegel disk or Herman Ring, then $\Lambda_+(x)=\mathcal J(R)$, where  $\mathcal J(R)$ is the Julia set of $R$. If $R$ is expansive on $\mathcal J(R)$ then $F$ is relatively hyperbolic on $\mathcal J(R)$. If $R$ is expanding on $\mathcal J(R)$, then $F$ is hyperbolic on $J(R)$ (see e.g., \cite[Chapter~19]{Mishasbook} for the definition of expanding and expansive rational maps).
 \end{ex}

Recall that if $F\:z\to w$ is invariantly inverse-like on $S$, we define the map $g_{F,S}\:F(S)\to S$ by $g_{F,S}(w)\coloneqq z$. Note that $F=g_{F,S}^{-1}$ on $S$.

%If $w$ is zero of multiplity $m$ of $P(z_0,w)$, where $P(z,w)=0$ defines $F$, then we say that $w$ has multiplicity $m$ in $F(z_0)$. 
 \begin{comment}
    
\begin{definition}\label{notattracting} Let $F$ be (relatively) hyperbolic on $\Lambda_+(x)$.
If $\omega\in\Omega_+(x)$, we say that a pinched neighborhood $U$ of $\Lambda_+(x)$, pinched at $\omega$, is pinched \emph{exactly in the repelling directions} if 
\begin{enumerate}
    \item for any sufficiently small neighborhood $V$ of $\omega$, $T_\omega^{n}$ is defined in $\overline {U\cap V}$ for all $n\geq 1$  and $T^n_\omega(z)\to \omega$ for all $z\in \overline {U\cap V}$, and 
    \item for all $z$ sufficiently close to $\omega$ such that for every sufficiently small neighborhood $V$ containing $z$ and $\omega$, $T_\omega^{-n}(z)\notin V$ for some $n\geq 1$, then $z\in U\cap V$. 
\end{enumerate} 
    
\end{definition}
\end{comment}
 The following definition will be important for our results. In essence, it gives us good control of the branches of $F^n(z)$ for each $n\geq 0$ and $z\in \Lambda_+(x)\smallsetminus \Omega_+(x)$.
\begin{definition}\label{def:minimal}
    Suppose that $\Lambda_+(x)$ is a forward limit set on which $F$ is invariantly inverse-like. We say that $\Lambda_+(x)$ is \emph{minimal} if the following conditions hold:
    \begin{enumerate}
        \smallskip \item\label{it:min1} $ \Lambda_+(x)$ has empty interior, $\Lambda_+(x)$ has no proper closed subset which is forward invariant, and $\Lambda_+(x)\cap(\sing\smallsetminus\mathcal{CV}_{F^{-1}})=\emptyset$.
       \smallskip  \item\label{it:U} For each $z\in \Lambda_+(x)\smallsetminus \Omega_+(x)$, there exists a neighborhood $U_z\subseteq \setRS$ of $z$ such that for each $n\geq 0$, $g_{F,\Lambda_+(x)}^n$ extends to a map (still denoted $g_{F,\Lambda_+(x)}^n$) on $F^n(U_z\cup \Lambda_+(x))$ so that $g_{F,\Lambda_+(x)}^n(F^n(y))=\{y\}$, $\Lambda_+(y)\subseteq \Lambda_+(x)$ for each $y\in U_z$, and there are at least three points in $\setRS$ that do not belong to $\bigcup_{n=0}^{+\infty} F^{n}(U_z)$.
        \smallskip \item\label{it:countable}
        $\Omega_+(x)$ is finite. %\footnote{We expect that condition~(\ref{it:countable}) follows from the other assumptions, but do not currently have a formal proof of this statement.}
    \end{enumerate}
\end{definition}

 \begin{rem}
     Note that in Definitions~\ref{def:Omega-attracting}, \ref{def:wellbehaved}, and \ref{def:minimal}, we do not require that $\Omega_+(x)$ is nonempty. 
\end{rem}
\begin{rem}\label{remark:suff}
            If $F$ is invariantly inverse-like on $\Lambda_+(x)$ and there exists a pinched neighborhood $U\subsetneq\setRS$ of $\Lambda_+(x)$,
           pinched at $\Omega_+(x)$,
           such that $F(U)\subseteq U$, $g_{F,\Lambda_+(x)}$ extends to a map on $F( U)\cup \Lambda_+(x)$ so that for each $y\in U$, $g_{F,\Lambda_+(x)}(F(y))=\{y\}$, and for each $y\in U$, except maybe a finite set of $y$, we have $\Lambda_+(y)\subseteq \Lambda_+(x)$.  % and if $F(y)\cap \Lambda_+(x)\neq \emptyset$ for some $y\in U_z$, then $y\in \Lambda_+(x)$,
           then condition~(\ref{it:U}) holds. This will be useful in Section~\ref{sec:bullett}. By a \emph{pinched neighborhood} $U\subseteq \setRS$ of $\Lambda_+(x)$, pinched at $\Omega_+(x)$, we mean the closure of an open set $V $, such that $V$ has finitely many connected components, $V$ contains $\Lambda_+(x)\smallsetminus \Omega_+(x)$ but no points in $\Omega_+(x)$, and $U=\overline V$ contains $\Lambda_+(x).$
\end{rem}
\begin{rem}\label{re:1}
    If $F$ is invariantly inverse-like on $\Lambda_+(x)$ and $\Lambda_+(x)$ is minimal, it follows that if $F(y)\cap \Lambda_+(x)\neq \emptyset$ for some $y\in U_z$ as in condition~(\ref{it:U}), then $y\in \Lambda_+(x)$. 
\end{rem}

\section{Examples of minimal forward limit sets on which the correspondence is relatively hyperbolic}\label{sec:ex}
In this section, we show that many of  the LLMM correspondences that are matings between antirational maps and Hecke groups, as well as many of the Bullett--Penrose correspondences have minimal forward limit sets on which the correspondence is relatively hyperbolic. 

One can easily see that if $R$ is a hyperbolic or parabolic rational map, then $F(z)\coloneqq R^{-1}(z)$ is invariantly inverse-like on the Julia set $J(R)$, which is a minimal forward limit set $\Lambda_+(x)$ for all except finitely many $x$. We focus here instead on correspondences that are not inverses of rational maps. We shall in particular, assuming Theorem~\ref{thm:existencemmeasure}, conclude Theorems~\ref{thm:misha} and~\ref{thm:BP}.

 \subsection{Forward limit sets of LLMM correspondences}\label{sec:misha}
In this subsection, we suppose that $f$ is a rational map of degree $d+1$ that is univalent on $\overline{\mathbb{D}}$. Denote by $\eta$ the reflection in the unit circle, i.e., $\eta (z) \coloneqq  1/\overline{z}$ for all $z \in \widehat{\mathbb{C}}$. The $d$-to-$d$ LLMM correspondence $F\:\setRS\to \setRS$ associated to $f$ is given by
\begin{equation}\label{def:lyubichmukherjeecorrespondence}
    F(z) \coloneqq  \biggl\{ w \in \widehat{\mathbb{C}} : \frac{f(w) -f(\eta (z))}{w -\eta (z)} =0 \biggr\}
\end{equation}
for all $z \in \widehat{\mathbb{C}}$. A version of these correspondences was introduced in \cite{Lee_2021} and mating results about them were obtained in \cite{lyubich2023antiholomorphic}.

The structure of this subsection is as follows. We first provide some preliminary results about $F$, including results from \cite{lyubich2023antiholomorphic}. We will then choose a certain class of LLMM correspondences, namely those defined by $f\in \mathcal M$, see Definition~\ref{def:finM}, and verify Definitions~\ref{def:wellbehaved} and \ref{def:minimal}. Lastly, we introduce an extra assumption, regarding the existence of attracting periodic points of $F^{-1}$ in a certain subset of $\setRS$, allowing us to apply Theorem~\ref{thm:existencemmeasure} to conclude Theorem~\ref{thm:misha}. This is the main result of this subsection. 

So, as indicated above, let us now introduce some preliminary results about $F$.
As $f$ is injective on $\overline{\mathbb{D}}$, $F(\mathbb{D}^*)\subseteq \mathbb{D}^*$. Take also a point $z\in \partial \mathbb{D}$ not equal to a critical point of $f$. Then $F(z)\cap \partial \mathbb{D}=\emptyset$, so the polynomial defining $F$ has no irreducible factors of the form $(z-aw)$ or $(zw-a)$, with $a$ being a root of unity. 
We define the Schwartz reflection and relate it to the correspondence $F$ following \cite[Section~2]{lyubich2023antiholomorphic}. We define $U \coloneqq f(\mathbb{D})$ and the Schwartz refection $\sigma \colon \overline{U} \rightarrow \widehat{\mathbb{C}}$ as $\sigma \coloneqq  f \circ \eta \circ (f|_{\overline{\mathbb{D}}})^{-1}$. Denote $\mathbb{D}^* \coloneqq  \widehat{\mathbb{C}} \smallsetminus \overline{\mathbb{D}}$. It is not difficult to check the following relation between $F$ and $\sigma$:
\begin{enumerate}
    \smallskip
    \item For $z \in \overline{\mathbb{D}}$, with $\eta(z)$ not a critical point of $f$, we have $w \in F(z)$ if and only if $w \neq \eta (z)$ and $f(w) =\sigma (f(z))$. If $\eta(z)$ is a critical point of $f$, then $w\in F(z)$ if and only if $f(w) =\sigma (f(z))$.
    \smallskip 
    \item For $z \in \mathbb{D}^*$, we have that $w \in F (z)$ only if $\sigma (f(w)) =f(z)$.
\end{enumerate}

%Since $f$ is injective on $\overline{D}$, for each $z \in \partial \mathbb{D}$, exactly one of the following three properties holds:
%\begin{enumerate}
%    \smallskip
%    \item $f' (z) \neq 0$, $f(w) \neq f(z)$ for all $w \in \partial \mathbb{D} \smallsetminus \{ z \}$.
%    \smallskip
%    \item $f' (z) =0$, $f'' (z) \neq 0$, $f(w) \neq f(z)$ for all $w \in \partial \mathbb{D} \smallsetminus \{ z \}$.
%    \smallskip
%    \item $f'(z) \neq 0$, there exists a unique $w \in \partial \mathbb{D} \smallsetminus \{ z \}$ such that $f(w) =f(z)$, and this $w$ satisfies $f' (w) \neq 0$.
%\end{enumerate}
%If we denote by $P_0,\, P_1$, and $P_2$ the set of points in $\partial \mathbb{D}$ satisfying (1), (2), and (3), respectively, then $P_0$, $P_1$, and $P_2$ form a partition of $\partial \mathbb{D}$.

Denote by $P_1$ the set of critical points of $f$ on $\partial \mathbb{D}$. For $z \in P_1$, we call $f(z)$ a \emph{cusp} of $\partial U$. %For $z \in P_2$, we call $f(z)$ a \emph{double point} of $\partial U$.
We denote by $S(\sigma)$ the set of all cusps of $\partial U$. Then $f^{-1} (S(\sigma)) \cap \partial \mathbb{D} =P_1$.
We then set $T^0 (\sigma) \coloneqq  \widehat{\mathbb{C}} \smallsetminus (U \cup S(\sigma))$ and
\begin{equation*}
    T^\infty (\sigma) \coloneqq  \bigcup_{n = 0}^{+\infty} \sigma^{-n} \bigl( T^0 (\sigma) \bigr).
\end{equation*}
We call $T^\infty (\sigma)$ the \emph{tiling set} of $\sigma$. The \emph{non-escaping set} of $\sigma$ is given by $K(\sigma) \coloneqq  \widehat{\mathbb{C}} \smallsetminus T^\infty (\sigma)$. By \cite[Proposition~2.2]{lyubich2023antiholomorphic}, $T^\infty (\sigma)$ is open and $K(\sigma)$ is closed in $\widehat{\mathbb{C}}$. %Then the \emph{limit set} of $\sigma$ is defined as the common boundary of $T^\infty (\sigma)$ and $K(\sigma)$, denoted by $\Lambda (\sigma)$.

For $z \in \partial U$, we have $\sigma (z) =z$. Consequently, for $z \in S(\sigma)\ (\subseteq \partial U)$, we have $\sigma^n (z) \notin T^0 (\sigma)$ for all $n \geq 0$, so $z \notin T^\infty (\sigma)$. Thus, $S(\sigma) \subseteq K(\sigma)$. %Since $S(\sigma) \subseteq \partial U$ and $\widehat{\mathbb{C}} \smallsetminus \overline{U} \subseteq T^0 (\sigma) \subseteq T^\infty (\sigma)$, we get $S(\sigma) \subseteq \Lambda (\sigma)$.
Since $\partial U \smallsetminus S(\sigma) \subseteq T^0 (\sigma) \subseteq T^\infty (\sigma)$, we have $S(\sigma) =K(\sigma) \cap \partial U$.

We write $\widetilde{T^\infty (\sigma)} \coloneqq  f^{-1} (T^\infty (\sigma))$, $\widetilde{K (\sigma)} \coloneqq  f^{-1} (K (\sigma))$, $K_+ \coloneqq  \widetilde{K (\sigma)} \cap \overline{\mathbb{D}^*}$, and $K_- \coloneqq  \widetilde{K (\sigma)} \cap \overline{\mathbb{D}}$. %Then \cite[Proposition~2.5]{lyubich2023antiholomorphic} gives the structure of $F$ on $\widetilde{K (\sigma)}$.
%\begin{prop}
%    (\cite[Proposition~2.5]{lyubich2023antiholomorphic}) The branches of $F$ on $\widetilde{K (\sigma)} =K_+ \bigcup K_-$ satisfy the following properties:
%    \begin{enumerate}
%        \smallskip
%        \item $F(K_+) =K_+$ and $F^{-1}(K_-)=K_-$.
%        \smallskip
%        \item $F$ has a forward branch carrying $K_-$ onto itself with degree $d$, and this branch is topologically conjugate to $\sigma \colon K(\sigma) \rightarrow K(\sigma)$. The remaining forward branches of $F$ carry $K_-$ onto $K_+$.
%        \smallskip
%        \item $F$ has a backward branch carrying $K_+$ onto itself with degree $d$, and this branch is topologically conjugate to $\sigma \colon K(\sigma) \rightarrow K(\sigma)$.
%    \end{enumerate}
%\end{prop}
Recall $S(\sigma) =K(\sigma) \cap \partial U$, $f(\partial \mathbb{D}) =\partial U$, and $f^{-1} (S(\sigma)) \cap \partial \mathbb{D} =P_1$. Then $\widetilde{K(\sigma)} \cap \partial \mathbb{D} =P_1$, so $K_+ \cap \partial \mathbb{D} =K_- \cap \partial \mathbb{D} =P_1$. %We claim
%\begin{equation}\label{etaK_+=K_-}
%    \eta (K_+) =K_-.
%\end{equation}
%Indeed, if $z \in K_- \subseteq \overline{\mathbb{D}}$, then $f(z) \in K(\sigma)$, so $\sigma^n (f(z)) \notin T^0 (\sigma)$ for all $n \geq 0$. This yields $\sigma (f(z)) \in K(\sigma)$. Recall $\sigma =f\circ \eta \circ (f|_{\overline{\mathbb{D}}})^{-1}$, so $\sigma (f(z)) =f(\eta (z))$. Consequently, $\eta (z) \in f^{-1} (K(\sigma)) =\widetilde{K(\sigma)}$. Since $z \in \overline{\mathbb{D}}$, we have $\eta (z) \in \overline{\mathbb{D}^*}$, so $\eta (z) \in \widetilde{K(\sigma)} \cap \overline{\mathbb{D}^*} =K_+$. Therefore, $\eta (K_-) \subseteq K_+$. Now let $z \in K_+$ be arbitrary. Then $f(z) \in K(\sigma)$, and thus $\sigma^n (f(z)) \notin T^0 (\sigma)$ for all $n \geq 0$. Since $z \in \overline{\mathbb{D}^*}$ and $\sigma =f\circ \eta \circ (f|_{\overline{\mathbb{D}}})^{-1}$, we have $f(z) =\sigma (f (\eta (z)))$, so $\sigma^n (f(\eta (z))) =\sigma^{n-1} (f(z)) \notin T^0 (\sigma)$ for all $n \geq 1$. Moreover, since $z \in K_+ \subseteq \mathbb{D}^* \bigcup P_1$, we have $f(\eta (z)) \in U \bigcup S(\sigma) =\widehat{\mathbb{C}} \smallsetminus T^0 (\sigma)$. Consequently, $f(\eta (z)) \in K(\sigma)$. Since $z \in \overline{\mathbb{D}^*}$, we have $\eta (z) \in \overline{\mathbb{D}}$, and thus $\eta (z) \in \widetilde{K(\sigma)} \cap \overline{\mathbb{D}} =K_-$. This shows that $\eta (K_+) \subseteq K_-$. Therefore, (\ref{etaK_+=K_-}) is established.
We set $\Lambda_+ \coloneqq  \partial K_+$ and $\Lambda_- \coloneqq  \partial K_-$. %Then by (\ref{etaK_+=K_-}), we have $\eta (\Lambda_+) =\Lambda_-$.
Since $P_1$ is finite, and thus is discrete, we also have $\Lambda_+ =\partial \widetilde{K(\sigma)} \cap \overline{\mathbb{D}^*}$ and $\Lambda_- =\partial \widetilde{K(\sigma)} \cap \overline{\mathbb{D}}$.
Since $K_+ \subseteq \overline{\mathbb{D}^*}$ and $K_+ \cap \partial \mathbb{D} =P_1$, we have $P_1 \subseteq \Lambda_+$. We let $\Omega_+$ denote the set of all periodic parabolic points of $F$ on $\Lambda_+$. 

%We recall the two criteria \cite[Propositions~2.15 and~2.19]{lyubich2023antiholomorphic} that make $F$ a mating:

%Bers like antirational map: it has a simply connected completely invariant Fatou component

%Suppose $T^\infty (\sigma)$ is a simply connected domain either containing no critical value of $f$ or containing exactly one critical value $v_0 \in T^0 (\sigma)$ of $f$ with $f^{-1} (v_0)$ a singleton. If, moreover, a pinched antipolynomial-like\footnote{See \cite[Definition~2.8]{lyubich2023antiholomorphic}} restriction of a Bers-like antirational map $R$ is hybrid equivalent\footnote{See \cite[Definition~2.9]{lyubich2023antiholomorphic}} to $\sigma$, then $F$ is a mating of $R$ and the group $(\mathbb{Z}/2\mathbb{Z}) * (\mathbb{Z}/(d+1)\mathbb{Z})$.

Next, denote $W \coloneqq  \overline{f^{-1} (U) \smallsetminus \mathbb{D}}$ and note that $W \cap \partial \mathbb{D} =P_1$. For $w \in W$, there exists exactly one point $w' \in \overline{\mathbb{D}}$ with $f(w') =f(w)$, so there exists exactly one $z_0 (=\eta (w')) \in \overline{\mathbb{D}^*}$ with $f(w) =f(\eta (z_0))$. As $\eta(z_0)\neq w$ if $z_0\notin P_1$, it follows that $w\in F(z_0)$. Moreover, $z_0$ is the only point in $\overline{\mathbb{D}^*}$ with $w \in F(z_0)$, i.e., $F^{-1} (w) \cap \overline{\mathbb{D}^*} = \{z_0\}$. %For $w \in P_2$, we have $F^{-1} (w) \cap \overline{\mathbb{D}^*} =\{w'\}$, where $w'$ is the only point on $\partial \mathbb{D} \smallsetminus \{w\}$ with $f(w') =f(w)$.
 This allows us to define an antiholomorphic map $g \colon W \rightarrow \overline{\mathbb{D}^*}$ such that $F^{-1} (w) \cap \overline{\mathbb{D}^*} =\{g(w)\}$ for all $w \in W$, i.e., $g$ is the only inverse branch of $F$ from $W$ to $\overline{\mathbb{D}^*}$.
%For $z \in \mathbb{D}^*$, if $w \in F(z)$, then $w \neq \eta (z)$ and $f(w) =f(\eta (z)) \in U$, so $w \in f^{-1} (U) \smallsetminus \mathbb{D}$. This implies $F(\mathbb{D}^*) \subseteq f^{-1} (U) \smallsetminus \mathbb{D}$. Then the construction of $g$ yields $F(\mathbb{D}^*) =f^{-1} (U) \smallsetminus \mathbb {D}$ and $g(f^{-1} (U) \smallsetminus \mathbb{D}) =\mathbb{D}^*$.
In particular, $g(w) =w$ for all $w \in P_1$.
Moreover, by definition of $\sigma$, we have
\begin{equation}\label{sigmafz=fgz}
    \sigma (f(z)) =f(g(z))
\end{equation}
for all $z \in \mathbb{D}^* \cup P_1$, such that both sides of (\ref{sigmafz=fgz}) are defined. That is, $\sigma$ is defined on $\overline{U}$ and $g$ is defined on $W$ and the equality~(\ref{sigmafz=fgz}) means that the two sides of (\ref{sigmafz=fgz}) are defined simultaneously, and if they are defined, then they are equal. %\Nils{It seems like here you use the relation between $F$ and $\sigma$ on the boundary of $\partial D$ so where you have (1) and (2) above, maybe one of them should consider also the boundary?} $z \in \mathbb{D}^* \bigcup P_1$. The argument above also implies
%\begin{enumerate}
%    \item $g(z) =z$ for all $z \in P_1$.
%    \smallskip
%    \item For each $z \in P_2$, we have $g(z) \in P_2$ and $g^2 (z) =z$.
%\end{enumerate} 
For $z \in (\mathbb{D}^* \cup P_1) \smallsetminus W$, suppose that $w \in F^{-1} (z) \cap \overline{\mathbb{D}^*}$, then $f(z) =f(\eta (w))$. Since $f^{-1} (\overline{U}) =\overline{D} \cup W$, we have $f(z) \notin \overline{U}$, but $f(\eta (w)) \in f(\overline{\mathbb{D}}) =\overline{U}$. This contradicts $f(z) =f(\eta (w))$, so $F^{-1} (z) \cap \overline{\mathbb{D}^*}$ is empty. Additionally, for $z \in \overline{ \mathbb{D}} \smallsetminus P_1$, if $w \in F^{-1} (z)$, i.e., $z \in F(w)$, then $f(z) =f(\eta (w))$ and $z \neq \eta (w)$. Consequently, we have $\eta (w) \in \mathbb{D}^*$ because $f$ is injective on $\overline{\mathbb{D}}$. Thus, $w \in \mathbb{D}$ and $F^{-1} (z) \cap \overline{\mathbb{D}^*} =\emptyset$. Therefore, $F^{-1} (z) \cap \overline{\mathbb{D}^*}$ is empty for all $z \in \widehat{\mathbb{C}} \smallsetminus W$.

For $z \in \mathbb{D}^* \cup P_1$, we have $f(z) \in K(\sigma)$ if and only if $\sigma^n (f(z))$ is defined for all $n \in \bN$. The equality (\ref{sigmafz=fgz}) implies 
\begin{equation}\label{eq:bothsides}
    \sigma^n (f(z)) =f(g^n (z))
\end{equation} for each $z \in \mathbb{D}^* \cup P_1$ such that both sides of \eqref{eq:bothsides} are defined. Consequently, for $z \in \widehat{\mathbb{C}}$, we have $z \in K_+$ if and only if $g^n (z)$ is defined and $g^n (z) \in W$ for all $n \in \bN$. This also shows that $g^n(z)\in K_+$ for all $n\in \bN$.%This implies $g(z) \in K_+$ when $z \in K_+$ and $g(z) \notin K_+$ when $z \in W \smallsetminus K_+$. Note that $f^{-1} (U) \smallsetminus \mathbb{D}$ is open in $\widehat{\mathbb{C}}$. By the continuity of $g$, when $z \in \Lambda_+ \cap (f^{-1} (U) \smallsetminus \mathbb{D})$, we have $g(z) \in \Lambda_+$. If $z \in \partial (f^{-1} (U) \smallsetminus \mathbb{D}) \cap \Lambda_+ \subseteq \partial (f^{-1} (U) \smallsetminus \mathbb{D}) \cap K_+$, then $g(z) \in \partial \mathbb{D} \cap W =P_1 \subseteq \Lambda_+$.

We summarize the arguments above into the following useful lemma.

\begin{lemma}\label{pandingK+}
    For each $z \in \widehat{\mathbb{C}}$, the following statements hold:
    \begin{enumerate}
        \smallskip\item\label{it:le641} $g(z)$ is defined if and only if $F^{-1} (z) \cap \overline{\mathbb{D}^*}$ is nonempty. Moreover, if $g(z)$ is defined, then $F^{-1} (z) \cap \overline{\mathbb{D}^*} =\{ g(z) \}$.
        \smallskip \item\label{it:le642} $z \in K_+$ if and only if $g^n (z)$ is defined for all $n \in \bN$.
    \end{enumerate}
\end{lemma}

Moreover, we claim
\begin{equation}\label{sigmafetaz=fetagz}
    \sigma (f (\eta (z))) =f ( \eta (g(z)))
\end{equation}
for all $z \in W$. Indeed, if $z \in W$, then $f(z) =f(\eta (g(z)))$. Since $\eta (z) \in \overline{\mathbb{D}}$, by the definition of $\sigma$, we have $\sigma (f(\eta (z))) =f(\eta (\eta (z))) =f(z)$, so \eqref{sigmafetaz=fetagz} holds and the claim is established. Since $\eta (W) \subseteq \overline{\mathbb{D}}$ and $f$ is injective on $\overline{\mathbb{D}}$, this claim implies that $g$ on $W$ and $\sigma$ on $f(\eta (W))$ are anti-conformally conjugate via the map $f \circ \eta$.

Recall from \cite[Proposition~2.4]{lyubich2023antiholomorphic} that $\widetilde{K(\sigma)} =K_+ \cup K_-$ is invariant under both $F$ and $F^{-1}$, so $\Lambda_+ \cup \Lambda_- =\partial \widetilde{K(\sigma)}$ is also invariant under both $F$ and $F^{-1}$. We can moreover summarize from \cite[Proposition~2.4]{lyubich2023antiholomorphic}, together with some simple calculations, the following information about the dynamics of $F$ and $F^{-1}$ on $\widetilde{K(\sigma)}$:
\begin{equation}\label{gLambda_+=Lambda_+}
    \quad g(K_+) =K_+,  \quad g (\Lambda_+) =\Lambda_+,
\end{equation}
\begin{equation}\label{F^-1zsubseteqg(z)cupK-}
    F^{-1} (z) \subseteq g(z) \cup (K_ -\smallsetminus P_1) \text{ for each } z \in K_+,
\end{equation}
\begin{equation}\label{F^-1zsubseteqg(z)cupLambda-}
    F^{-1} (z) \subseteq g(z) \cup (\Lambda_-\smallsetminus P_1) \text{ for each } z \in \Lambda_+,
\end{equation}
\begin{equation}\label{F^-1(K-)=K-F^-1(Lambda-)=Lambda-,F(K+)=K+,F(Lambda+)=Lambda+}
    F^{-1} (K_-) =K_-,\quad F^{-1} (\Lambda_-) =\Lambda_-,\quad F (K_+) =K_+,\quad F (\Lambda_+) =\Lambda_+.
\end{equation}

Additionally, \cite[Proposition~2.4]{lyubich2023antiholomorphic} also implies
\begin{equation}\label{etaK_+=K_-}
    \eta (K_+) =K_- \quad \text{and} \quad \eta (\Lambda_+) =\Lambda_-.
\end{equation}
It now follows by \eqref{gLambda_+=Lambda_+}, \eqref{F^-1zsubseteqg(z)cupK-}, and the last two equations of \eqref{F^-1(K-)=K-F^-1(Lambda-)=Lambda-,F(K+)=K+,F(Lambda+)=Lambda+}, that $F$ is invariantly inverse-like on $K_+$ and on $\Lambda_+$.

Recall from \cite[Section~1]{lyubich2023antiholomorphic} that an antirational map $R$ of degree strictly greater than 1 is called \emph{Bers-like} if it has a simply connected completely invariant Fatou component. We will assume that a Bers-like antirational map is equipped with a marked simply connected completely invariant Fatou component $\mathcal{B} (R)$ (in case there are more than one). A \emph{polygon} is a Jordan domain whose boundary consists of finitely many smooth arcs. We now define the class of rational maps $f$ which we will be working with. 

%Now we recall some notions from \cite[Section~2.4]{lyubich2023antiholomorphic}.

%We say that a \emph{polygon} is a Jordan domain whose boundary consists of finitely many closed smooth arcs. The points of intersection of these arcs will be denoted as the \emph{corners} of the polygon. A \emph{pinched polygon} is a union of domains in $\widehat{\mathbb{C}}$ whose closure is homeomorphic to a closed disk quotiented by a finite geodesic lamination, and whose boundary is given by finitely many closed smooth arcs. The separating points of the closure of a pinched polygon will be called its \emph{pinchings}.

%\begin{definition}
%    Let $V \subseteq \widehat{\mathbb{C}}$ be a polygon, and $U \subseteq V$ be a pinched polygon with $\overline{U} \cap \partial V$ consisting of all corners of $V$.
%\end{definition}

\begin{definition}\label{def:finM}
    We define $\mathcal{M}$ as the class of all degree $d+1$, with $d\geq 1$, rational maps $f \colon \widehat{\mathbb{C}} \rightarrow \widehat{\mathbb{C}}$ that are univalent on $\overline{\mathbb{D}}$, with the following properties:
    \begin{enumerate}
        \smallskip
        \item\label{it:defM1} $T^\infty (\sigma)$ is a simply connected domain either containing no critical value of $f$ or containing exactly one critical value $v_0 \in T^0 (\sigma)$ of $f$, with $f^{-1} (v_0)$ a singleton.
        \smallskip
        \item\label{it:defM2} There exists a Bers-like antirational map $R$ with no Fatou component being a Siegel disk or a Herman ring %with no critical points on its Julia set\footnote{maybe this implies (3).}
        and a forward invariant Jordan domain $D \subsetneq \mathcal{B} (R)$ such that %$\bigcup_{n=0}^{+\infty} R^{-n} (D) =\mathcal{B} (R)$ and that 
        $V \coloneqq  \widehat{\mathbb{C}} \smallsetminus \overline{D}$ is a polygon.  Let $\nu (V)$ denote the set of corners of $V$. Furthermore, there exists a homeomorphism $\Phi$ from a pinched neighborhood of $K(\sigma)$, pinched at $S(\sigma) \cup \sigma^{-1} (S(\sigma))$ to a pinched neighborhood of $\widehat{\mathbb{C}} \smallsetminus \mathcal{B} (R)$, pinched at $\nu (V) \cup R^{-1} (\nu (V))$ with the following properties:
        \begin{enumerate}
            \smallskip\item $\Phi (S(\sigma)) =\nu (V)$.
            \smallskip\item $\Phi (K(\sigma)) =\widehat{\mathbb{C}} \smallsetminus \mathcal{B} (R)$.
            \smallskip\item $\Phi$ conjugates $\sigma$ to $R$ on its domain.
        \end{enumerate}
        %pinched antipolynomial-like maps\footnote{See \cite[Definition~2.8]{lyubich2023antiholomorphic}. In particular, $S(\sigma)$ is the set of all corners of $\overline{U}$, and $S(\sigma) \bigcup \sigma^{-1} (S(\sigma))$ is the set of all corners and pinchings of $\sigma^{-1} (\overline{U})$.} $(R|_{\overline{R^{-1} (V)}}, \overline{R^{-1} (V)}, \overline{V})$ and $(\sigma|_{\overline{\sigma^{-1} (U)}}, \overline{\sigma^{-1} (U)}, \overline{U})$ are hybrid equivalent\footnote{See \cite[Definition~2.9]{lyubich2023antiholomorphic}. In particular, this implies that there exists a bijection $\Phi$ conjugating $\sigma$ to $R$ on the closure of a neighborhood of $K(\sigma)$ pinched at the corners and pinchings of $\sigma^{-1} (\overline{U})$, and sending all corners of $\overline{U}$ to all corners of $\overline{V}$.} by a topological conjugation map that conjugates their filled Julia sets\footnote{I don't know if this property (filled Julia sets are conjugated) holds automatically because Sabya's definition of hybrid equivalence is vague. I hope it holds.}.
        \smallskip
        \item\label{it:defM3} If $z \in \Lambda_+$ is a critical point of $f$, then $z \in P_1$.
    \end{enumerate}
\end{definition}

Then, by \cite[Propositions~2.15 and~2.19]{lyubich2023antiholomorphic}, if $f \in \mathcal{M}$, and the homeomorphism $\Phi$ in Definition~\ref{def:finM}~(\ref{it:defM2}) is quasiconformal (resp.\ David) (see for example, \cite[Definition~3.1]{lyubich2023antiholomorphic}), and antiholomorphic a.e.\ on $K(\sigma)$, then the correspondence $F$ given by (\ref{def:lyubichmukherjeecorrespondence}) is a quasiconformal (resp.\ David) mating of $R$ and the group $(\mathbb{Z}/2\mathbb{Z}) * (\mathbb{Z}/(d+1)\mathbb{Z})$. We will now establish some basic properties and then check Definitions~\ref{def:Omega-attracting},~\ref{def:wellbehaved}, and~\ref{def:minimal} for $F$ when $f \in \mathcal{M}$. First, we have that $\Lambda_+$ is a forward limit set.

\begin{lemma}\label{le:islimitset}
Let $f\in \mathcal M$.
    For all but finitely many $x\in \operatorname{int}(K_+)$, we have $\Lambda_+(x)=\Lambda_+$.
\end{lemma}
\begin{proof}\label{le:limitset}
Take $y\in \setRS\smallsetminus \mathcal{B} (R)$ not equal to an attracting periodic point of $R$ and recall that any rational map has only finitely many attracting periodic points. Then $\bigcap _{n=0}^{+\infty} \overline{\bigcup_{k=n}^{+\infty} R^{-k}(x)}=J(R)$. Since $\Phi \circ f \circ \eta$ from $K_+$ to $\widehat{\mathbb{C}} \smallsetminus \mathcal{B} (R)$ conjugates $g$ to $R$, and that $F =g^{-1}$ in $K_+$, it follows that if $x=(\Phi \circ f \circ \eta)^{-1}(y)$, then $\Lambda_+(x)=\Lambda_+$.
\end{proof}

Throughout the rest of this section, we suppose that we have chosen $x\in \setRS$ so that $\Lambda_+(x)=\Lambda_+$ and $\Omega_+=\Omega_+(x)$. 

We have the following lemma.

\begin{prop}\label{Omega=P_1}
    Let $f \in \mathcal{M}$. Then $P_1 \subseteq \Omega_+$. Moreover, for each $z_0 \in P_1$, the branch of $F$ fixing $z_0$ has the form $z \mapsto 2z_0 -\eta (z) +o(|z-z_0|)$ for $z$ close enough to $z_0$.
\end{prop}

\begin{proof}
    %First, $\Omega_+(x)$ is nonempty\footnote{This is from Sabya's letter.}.
    %We first show that $P_1 \subseteq \Omega_+$% and $\Lambda_+ \cap \partial \mathbb{D} =P_1 \cap P_2$
    %.
    Fix an arbitrary $z_0 \in P_1$. Suppose
    \begin{equation}\label{f(y)-f(x) =(y-x)h(x,y)}
        f(y)-f(x) =(y-x)h(x,y),
    \end{equation}
    where $h$ is a rational map. We take partial derivatives on both sides of (\ref{f(y)-f(x) =(y-x)h(x,y)}) and then substitute $x=z_0$ and $y=z_0$, yielding $h (z_0, z_0)= f' (z_0) =0$. Next, taking twice the partial derivatives on both sides of (\ref{f(y)-f(x) =(y-x)h(x,y)}) and then substituting $x=z_0$ and $y=z_0$ gives $h_x (z_0, z_0)= f'' (z_0) /2 \neq 0$. The same argument shows $h_y (z_0, z_0) =f'' (z_0) /2 \neq 0$. By definition of $F$, each branch $f$ of $F$ satisfies $h(\eta (z), f (z)) =0$. Since $z_0 \in \partial \mathbb{D}$ is a quadratic critical point, there exists a unique branch $f$ of $F$ defined on a neighborhood of $z_0$ and fixing $z_0$. Moreover, we have $\frac{\partial f}{\partial \eta (z)} \big|_{z=z_0} =-\frac{h_x (z_0, z_0)}{h_y (z_0, z_0)} =-1$, by the implicit function theorem.  Therefore, $P_1 \subseteq \Omega_+(x)$ and $f (z) =2z_0 -\eta (z) +o(|z-z_0|)$ for $z$ close enough to $z_0$. 
    \end{proof}

We now show that $F$ is locally $\Omega_+(x)$-attracting on $\Lambda_+$.

\begin{prop}\label{locally attracting}
    If $f \in \mathcal{M}$, and $x$ is such that $\Lambda_+(x)=\Lambda_+$, then $F$ is locally $\Omega_+(x)$-attracting on $\Lambda_+(x)$.
\end{prop}

\begin{proof}
    To verify that $F$ is locally $\Omega_+(x)$-attracting on $\Lambda_+(x)$, we first investigate the branch of $F$ fixing a point in $P_1$, under the assumption that $P_1\neq \emptyset$. Without loss of generality, we assume $1 \in P_1$, suppose $F_1$ is the branch of $F$ defined on a small neighborhood $B_1$ of $1$ and fixing $1$, and aim to find a pinched neighborhood $U_1$ of $1$ and such that $F_1$ is contracting on $\Lambda_+(x) \cap U_1$ in the sense of Definition~\ref{def:Omega-attracting}. %The definition of $F_1$ can be simplified as follows: For $z \in B_1$, $F_1 (z)$ is the only point\footnote{The existence and uniqueness of this point is ensured by $f'(1)=0$ and $f'' (1) \neq 0$.} closed to $1$, different from $\eta (z)$, and satisfying $f(F_1 (z)) =f(\eta (z))$. From this definition of $F_1$,
    
    By Proposition~\ref{Omega=P_1}, we have
    \begin{equation*}
    \begin{aligned}
        F_1 (z) 
        &=2-\eta (z) +o(|z-1|) =1+(\overline{z} -1)/\overline{z} +o(|z-1|)\\
        &= 1+ (\overline{z} -1) +o(|z-1|) =\overline{z} +o(|z-1|).
    \end{aligned}
    \end{equation*}
    and thus $F_1^2 (z) =z+ o(|z-1|)$ for $z$ close enough to $1$. As $g$ is topologically conjugate to a rational map on $K_+$, which contains $1$, $F_1^2(z)$ is not equal to the identity map near $1$. Hence, there are $n$ attracting and $n$ repelling directions, for some $n\geq 1$, of $F_1^2$ at the parabolic fixed point $1$. Since $F_1^{2k+1} (z) =\overline{F_1^{2k} (z)} +o(|1-F_1^{2k} (z)|)$, we get that the direction $\mathbf{v}$ is attracting (resp., repelling) if and only if the mirroring direction of $\mathbf{v}$ about the normal direction of the unit circle at $=1$ is attracting (resp., repelling). As every repelling direction has two adjacent attracting directions and vice versa, this implies that $\mathbf{1}$ and $\mathbf{-1}$ are attracting or repelling directions.
Moreover, if $F_1 (z) =w$, then $f(\eta (z)) =f(w)$ and $\eta (z) \neq w$, so $F_1 (\eta (w)) =\eta (z)$. This implies that the direction $\mathbf{v}$ is attracting (resp., repelling) if and only if the mirroring direction of $\mathbf{v}$ about the tangent line of the unit circle at $1$ is repelling (resp., attracting). %\Nils{Here, the notation needs to be a bit clearer as $1/\overline{\mathbf{v}}={\mathbf{v}}$ viewed from the origin. Clarify that v starts from 1 and mirroring about the tangent line of the unit circle.} 
Consequently, $\mathbf{i}$ is not an attracting or repelling direction, and neither is $\mathbf{-i}$. When we say that a direction is contained in $\mathbb{D}$ we mean that if we draw a sufficiently short vector from the point $1$ along this direction, the vector will be contained in $\mathbb{D}$.
    Suppose $\mathbf{v}$ and $\mathbf{v}'$ are two attracting or repelling directions closest to the direction $\mathbf{i}$ contained in $\mathbb{D}$ and $\mathbb{D}^*$, respectively. Then by $F(\mathbb{D}^*) \subseteq \overline{\mathbb{D}^*}$, $\mathbf{v}$ is repelling and $\mathbf{v}'$ is attracting. The same argument holds for the direction $\mathbf{-i}$. Therefore, $n$ is odd, and possible dynamics of $F_1$ is as in Figure~\ref{fig:enter-label}.

    \begin{figure}
        \centering
        \includegraphics[scale=3]{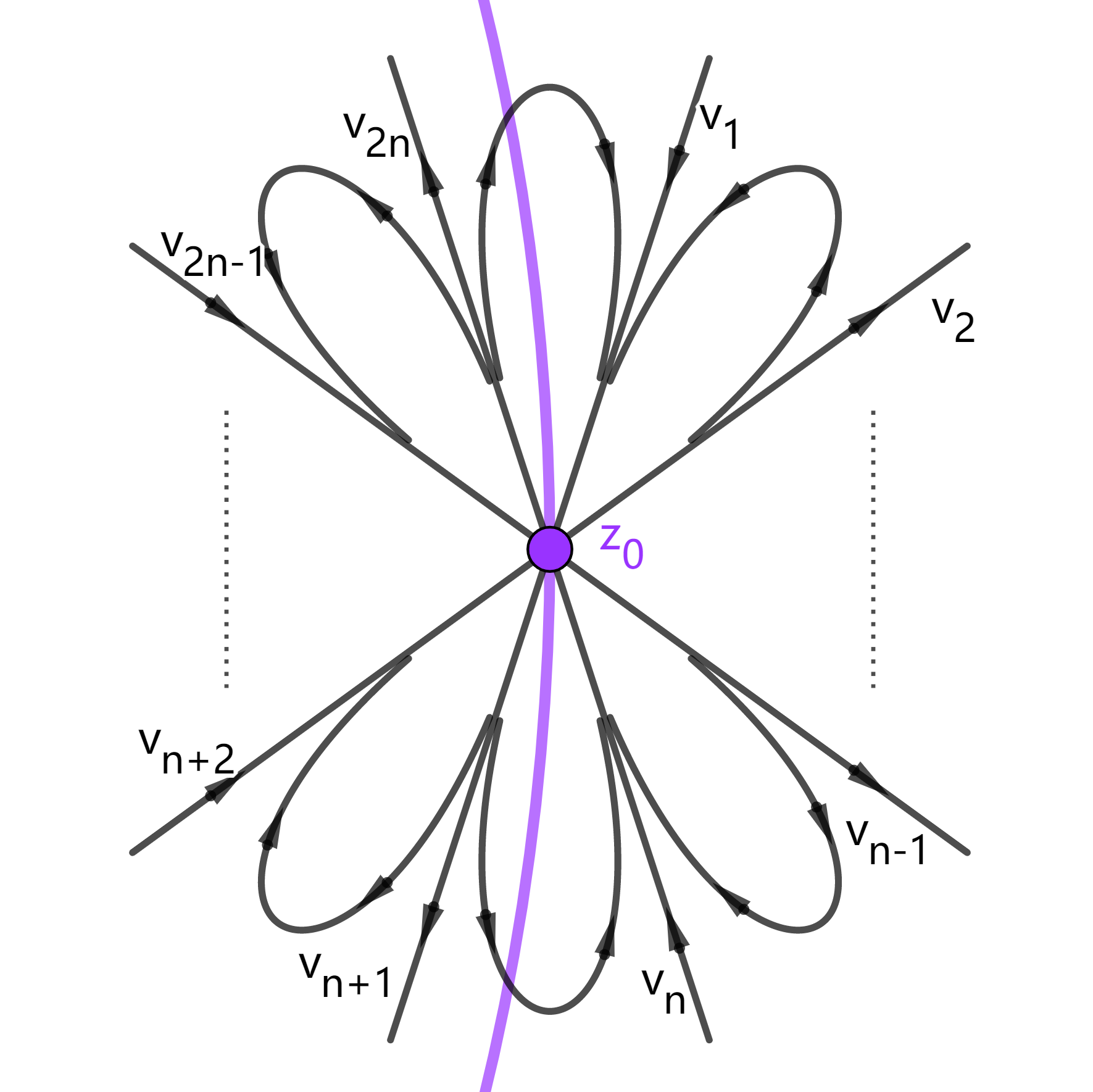}
        \caption{The dynamics of $F_1$ on a neighborhood of the parabolic fixed point $z_0 =1$. The purple curve is a part of the unit circle. $\mathbf{v}_1,\, \mathbf{v}_3,\, \dots ,\, \mathbf{v}_{2n-1}$ are all attracting directions, and $\mathbf{v}_2,\, \mathbf{v}_4,\, \dots ,\, \mathbf{v}_{2n}$ are all repelling directions.}
        \label{fig:enter-label}
    \end{figure}

    For $\alpha >0$, $\epsilon >0$, $z \in \mathbb{C}$, and a vector $\mathbf{v} \in \mathbb{C}$, denote
    \begin{equation*}
        B_{\epsilon, \alpha} (z, \mathbf{v}) \coloneqq \biggl\{ r e^{i \beta} \frac{\mathbf{v}}{\|\mathbf{v}\|} +z \colon 0<r< \epsilon,\, -\alpha <\beta <\alpha \biggr\},
    \end{equation*}
    i.e., $B_{\epsilon, \alpha} (z, \mathbf{v})$ is the intersection of the $\epsilon$-neighborhood of $z$ with the cone with the vertex $z$, the pole direction $\mathbf{v}$, and the field angle $2\alpha$. We use the ordering of $\mathbf{v}_j$ from Figure~\ref{fig:enter-label}.

    The dynamical behavior of $F_1$ near $z_0 =1$ and Lemma~\ref{pandingK+} implies the following:
    \begin{itemize}
        \smallskip\item For each $\delta >0$, there exists $\epsilon >0$ such that for every $j= 2,\, 4,\, \dots ,\, n-1$, there holds $B_{\epsilon, (\pi/n) -\delta} (1, \mathbf{v}_j) \subseteq K_+$.
        \smallskip\item For each $\delta >0$, there exists $\epsilon >0$ such that $B_{\epsilon, (\pi/2n) -\delta} (1, \mathbf{i}) \cap K_+ =\emptyset$ and $B_{\epsilon, (\pi/2n) -\delta} (1, \mathbf{-i}) \cap K_+ =\emptyset$.
    \end{itemize}
    Thus, for each $\delta >0$, we can choose $\epsilon =\epsilon (\delta) >0$ such that in the $\epsilon$-neighborhood of $1$, $\Lambda_+(x) =\partial K_+$ is contained in the pinched neighborhood
    \begin{equation}
        B^{\att}_{\delta} (1) \coloneqq \bigcup_{j=0}^{(n-1)/2} \overline{B_{\epsilon, \delta} (1, \mathbf{v}_{2j+1})},
    \end{equation}
    and, moreover, for each $w \in B^{\att}_{\delta} (1)$, $F_1^m (w) \to 1$ as $m \to +\infty$. Hence $F$ is locally $\Omega_+(x)$-attracting on $\Lambda_+(x)$ near the point 1, in the sense that there exists a pinched neighborhood $U_1$ of 1 satisfying the properties prescribed in Definition~\ref{def:Omega-attracting}. The same argument holds for all $z_0 \in P_1$.

    Now we fix an arbitrary $z_0 \in \Omega_+ \smallsetminus P_1$. We will use similar techniques as above to show that $F$ is locally $\Omega_+(x)$-attracting near $z_0$. Note that $z_0 \in \mathbb{D}^*$, so a neighborhood of $z_0$ is contained in $\mathbb{D}^*$, which is a significant difference between the parabolic periodic points not in $P_1$ and those in $P_1$. Suppose that a branch $T_{z_0}$ of $F^{2q}$ for some $q \in \bN$ fixes $z_0$. Again, $T_{z_0}$ is not topologically conjugate to a rational rotation for the same reasons as the case $z_0\in P_1$. Suppose there are $n$ attracting and $n$ repelling directions of $T_{z_0}$ at $z_0$. Then Lemma~\ref{pandingK+} implies that for each $\delta >0$, there exists $\epsilon >0$ such that for every repelling direction $\mathbf{v}$ at $z_0$, there holds $B_{\epsilon, (\pi /n) -\delta} (z_0, \mathbf{v}) \subseteq K_+$. Thus, for each $\delta >0$, we can choose $\epsilon =\epsilon (\delta) >0$ such that in the $\epsilon$-neighborhood of $z_0$, $\Lambda_+(x) =\partial K_+$ is contained in the pinched neighborhood
    \begin{equation}
        B^{\att}_{\delta} (z_0) \coloneqq \bigcup_{\mathbf{v} \text{ is an attracting direction}} \overline{B_{\epsilon, \delta} (z_0, \mathbf{v})},
    \end{equation}
    and, moreover, for each $w \in B^{\att}_{\delta} (z_0)$, $T_{z_0}^n (w) \to z_0$ as $n \to +\infty$. Hence $F$ is locally $\Omega_+(x)$-attracting near $z_0$. In summary, we get that $F$ is $\Omega_+(x)$-attracting on $\Lambda_+(x)$.
\end{proof}

Next, we show the following proposition regarding the postcritical set of $F^{-1}$.

\begin{prop}\label{Lambda+capPCF-1=Omega}
    If $f \in \mathcal{M}$, then $\Lambda_+ \cap \mathcal{PC}_{F^{-1}} \subseteq \Omega_+$.
\end{prop}

\begin{proof}
    Note that if a point $z \in \widehat{\mathbb{C}}$ belongs to $\mathcal{CV}_{F^{-1}}$, then there exists some $w \in F(z)$ with $f' (w)=0$. Recall that $\mathcal{PC}_{F^{-1}} =\overline{\bigcup_{i=0}^{+\infty} F^{-i} (\mathcal{CV}_{F^{-1}})}$, so we have $\mathcal{PC}_{F^{-1}} \subseteq \overline{\bigcup_{i=1}^{+\infty} F^{-i} (\mathcal{C}_f)}$, where $\mathcal{C}_f$ denotes the set of all critical points of $f$. We will show $\Lambda_+ \cap \mathcal{PC}_{F^{-1}} \subseteq \Omega_+$ by showing
    \begin{equation}\label{9u99e99cs99d}
        \Lambda_+ \cap \overline{\bigcup_{i=1}^{+\infty} F^{-i} (\mathcal{C}_f)} \subseteq \Omega_+.
    \end{equation}

    Condition~(\ref{it:defM1}) in Definition~\ref{def:finM} implies that any critical point of $f$ is either in $f^{-1} (T^0 (\sigma))$ or in $f^{-1} (K(\sigma)) =K_+ \cup K_-$. For $z \in f^{-1} (T^0 (\sigma)) \cup \partial \mathbb{D}$, if $w \in F^{-1} (z)$, then $f(\eta (w)) =f(z) \in T^0 (\sigma) \cup \partial U = \widehat{\mathbb{C}} \smallsetminus U$. Then $\eta (w) \notin \mathbb{D}$, so $w \in \overline{\mathbb{D}}$. Since $F(\mathbb{D}^*) \subseteq \overline{\mathbb{D}^*}$, we have $F^{-1} (\mathbb{D}) \subseteq \overline{\mathbb{D}}$. Hence we get $F^{-1} (f^{-1} (T^0 (\sigma)) \cup \overline{\mathbb{D}}) \subseteq \overline{\mathbb{D}}$, and thus $F^{-n} (f^{-1} (T^0 (\sigma)) \cup \overline{\mathbb{D}}) \subseteq \overline{\mathbb{D}}$ holds for all $n \in \bN$. Since $\Lambda_+ \cap \overline{\mathbb{D}} =P_1$, we get that $\Lambda_+ \cap \overline{\bigcup_{i=1}^{+\infty} F^{-i} (z)} \subseteq P_1$ holds for all $z \in f^{-1} (T^0 (\sigma))$. Consequently, to prove \eqref{9u99e99cs99d}, it remains to show that for each critical point $z \in K_+ \cup K_-$ of $f$, it holds that
    \begin{equation}\label{9u99e988888}
        \Lambda_+ \cap \overline{\bigcup_{i=1}^{+\infty} F^{-i} (z)} \subseteq \Omega_+.
    \end{equation}

    Let $z \in K_+ \cup K_-$ be an arbitrary critical point of $f$. Since $f$ is univalent on $\mathbb{D}$, we have $z \in (K_+ \cup K_-) \smallsetminus \mathbb{D} =K_+$. Then condition~(\ref{it:defM3}) in Definition~\ref{def:finM} implies $z \in \operatorname{int} (K_+) \cup P_1 \subseteq W$.

    We argue by contradiction and assume that \eqref{9u99e988888} does not hold. This means we can choose a point $w \in \Lambda_+ \smallsetminus \Omega_+$ that belongs to $\overline{\bigcup_{i=1}^{+\infty} F^{-i} (z)}$. Since $w \in \Lambda_+ \smallsetminus \Omega_+ \subseteq \Lambda_+ \smallsetminus P_1 \subseteq \mathbb{D}^*$ and $K_- \subseteq \overline{\mathbb{D}}$, by \eqref{gLambda_+=Lambda_+}, \eqref{F^-1zsubseteqg(z)cupK-}, and $F^{-1} (K_-) =K_-$ in \eqref{F^-1(K-)=K-F^-1(Lambda-)=Lambda-,F(K+)=K+,F(Lambda+)=Lambda+}, we have $w \in \overline{ \{ g^i (z) : i \in \bN \}}$. Note that by \eqref{gLambda_+=Lambda_+} we have $\overline{ \{ g^i (z) : i \in \bN \}} \subseteq K_+ \subseteq W$. Then \eqref{sigmafetaz=fetagz} implies $f (\eta(w)) \in \overline{ \{ \sigma^i (f (\eta(z))) : i \in \bN \}}$, where $\overline{ \{ \sigma^i (f (\eta(z))) : i \in \bN \}} \subseteq f(\eta (K_+)) =f(K_-) \subseteq K(\sigma)$ and $f( \eta (w)) \in f( \eta (\Lambda_+)) =f(\Lambda_-) \subseteq \partial K(\sigma)$ by \eqref{etaK_+=K_-}.

    Since $K_+ \subseteq f^{-1} (U)$, $K_+ \smallsetminus P_1 \subseteq \mathbb{D}^*$, and $f^{-1} (U) \smallsetminus \mathbb{D} \subseteq W$, we have $K_+ \smallsetminus P_1 \subseteq \operatorname{int} (W)$. Since $w \in \Lambda_+ \smallsetminus P_1 \subseteq K_+ \smallsetminus P_1$ is not a parabolic periodic point of $F$ and $F(\Lambda_+) =\Lambda_+$, $w$ is not a parabolic periodic point of $g$. Then $f(\eta (w))$ is not a parabolic periodic point of $\sigma$ because $g$ is topologically conjugate to $\sigma$ on a neighborhood of $w$.

    Recall $z \in \operatorname{int} (K_+) \cup P_1$. Since $g^n (z') =z'$ for all $z' \in P_1$, $w \notin P_1$, and $w \in \overline{\{ g^i (z) : i \in \bN \}}$, it follows that $z \in P_1$. Consequently, $z \in \operatorname{int} (K_+)$. By \eqref{etaK_+=K_-}, $f(\eta (z)) \in \operatorname{int} (K(\sigma))$.

    %Since $z \in \operatorname{int} (K_+) \subseteq \mathbb{D}^*$, $z$ is a critical point of $f$, and $\sigma =f \circ \eta \circ (f|_{\overline{\mathbb{D}}})^{-1}$, we conclude that $f(\eta (z))$ is a critical point of $\sigma$.

    %there exists a Bers-like antirational map $R$ %with no critical points on its Julia set\footnote{maybe this implies (3).}
    %    and a forward invariant Jordan domain $D \subsetneq \mathcal{B} (R)$ such that $\bigcup_{n=0}^{+\infty} R^{-n} (D) =\mathcal{B} (R)$ and that $V \coloneqq  \widehat{\mathbb{C}} \smallsetminus \overline{D}$ is a polygon\footnote{A \emph{polygon} is a Jordan domain whose boundary consists of finitely many closed smooth arcs.}, and a homeomorphism $\Phi$ from the closure of a neighborhood of $K(\sigma)$ pinched at all points in $S(\sigma) \bigcup \sigma^{-1} (S(\sigma))$ to the closure of a neighborhood of $\widehat{\mathbb{C}} \smallsetminus \mathcal{B} (R)$ pinched at all corners and pinchings of $V$ with the following properties:
     %   \begin{enumerate}
      %      \item $\Phi (S(\sigma) \bigcup \sigma^{-1} (S(\sigma)))$ is the set of all corners and pinchings of $V$.
       %     \item $\Phi (K(\sigma)) =\widehat{\mathbb{C}} \smallsetminus \mathcal{B} (R)$.
        %    \item $\Phi$ conjugates $\sigma$ to $R$.
        %\end{enumerate}

    By condition~(\ref{it:defM2}) in Definition~\ref{def:finM}, we can choose a Bers-like antirational map $R$, %with no critical point on its Julia set
     a forward invariant Jordan domain $D \subsetneq \mathcal{B} (R)$ such that %$\bigcup_{n=0}^{+\infty} R^{-n} (D) =\mathcal{B} (R)$ and that
    $V \coloneqq  \widehat{\mathbb{C}} \smallsetminus \overline{D}$ is a polygon, and a homeomorphism $\Phi$ conjugating $\sigma$ to $R$ on the closure of a neighborhood of $K(\sigma)$ pinched at all points in $S(\sigma) \cup \sigma^{-1} (S(\sigma))$, such that $\Phi (S(\sigma)) =\nu (V)$ and $\Phi (K(\sigma)) =\widehat{\mathbb{C}} \smallsetminus \mathcal{B} (R)$.

    Write $z' =\Phi (f (\eta (z)))$ and $w' =\Phi (f (\eta (w)))$. Then $z' \in \operatorname{int} (\widehat{\mathbb{C}} \smallsetminus \mathcal{B} (R))$ as $f(\eta (z)) \in \operatorname{int} (K(\sigma))$. Moreover, $w' \in \overline{\{ R^i (z') : i \in \bN \}} \cap \Phi (\partial  K(\sigma))$. Since $\mathcal{B} (R)$ is completely invariant under $R$, %and $\widehat{\mathbb{C}} \smallsetminus \mathcal{B} (R) \subseteq V$, the filled Julia set $K(R|_{\overline{R^{-1} (V)}}) =\{ x \in \widehat{\mathbb{C}} : R^n (x) \in R^{-1} (V)$ for all $n \in \mathbb{Z}_{\geq 0} \}$ contains $\widehat{\mathbb{C}} \smallsetminus \mathcal{B} (R)$. Consequently, $w' \in \Phi (\partial K(\sigma)) =\partial K(R|_{\overline{R^{-1} (V)}}) \subseteq \overline{\mathcal{B} (R)} =\mathcal{B} (R) \bigcup J(R)$, where $J(R)$ refers to the Julia set of $R$. Recall $z'$ is in the Fatou set of $R$.
    the Julia set $J(R)$ of $R$ is the boundary of $\mathcal{B} (R)$. Consequently, $z' \in \operatorname{int} (\widehat{\mathbb{C}} \smallsetminus \mathcal{B} (R)) = (\widehat{\mathbb{C}} \smallsetminus \mathcal{B} (R)) \smallsetminus J(R)$ is in the Fatou set of $R$, and $w' \in \Phi (\partial K(\sigma)) =\partial \Phi (K (\sigma)) =\partial (\widehat{\mathbb{C}} \smallsetminus \mathcal{B} (R)) =J(R)$. 
    %Recall that $\mathcal{B} (R)$ is an attracting basin of either an attracting periodic point of $R$, or a parabolic periodic point of $R$. If $w' \in \mathcal{B} (R)$, then by $w' \in \overline{\{ R^i (z') : i \in \bN \}}$, there exists $n \in \bN$ such that $R^n (z') \in \mathcal{B} (R)$, which indicates $z' \in \mathcal{B} (R)$. The only possibility in this case is that $w'$ is an attracting periodic point and $\mathcal{B} (R)$ is the attracting basin of $w'$. But $w' \in K(R|_{\overline{R^{-1} (V)}}) \subseteq V =\widehat{\mathbb{C}} \smallsetminus \overline{D}$ implies $w' \notin \overline{D}$. This contradicts the fact that $D \subseteq \mathcal{B} (R)$ is forward invariant.
    %Hence $w' \in J(R)$, then 
    By Sullivan's classification theorem, $w' \in \overline{\{ R^i (z') : i \in \bN \}}$ implies that $w'$ is a parabolic periodic point of $R$.

    We first assume that $w' \in \nu (V) \cup R^{-1} (\nu (V))$. Since %the preimage of all corners of $\overline{V}$ is all corners and pinchings of $\overline{R^{-1} (V)}$ and
    $w'$ is periodic, $w' =\Phi (f (\eta (w)))$ must be in $\nu (V)$.  %Recall that $\mathcal{B} (R)$ is an attracting basin of either an attracting periodic point of $R$, or a parabolic periodic point of $R$. If the ultimately periodic point $w'$ belongs to $\mathcal{B} (R)$, then $\mathcal{B} (R)$ is an attracting basin of an attracting periodic point $y_0$, and there exists $n \in \bN$ with $R^n (w') =y_0$. Then $w' \in K(R|_{\overline{R^{-1} (V)}})$ implies $y_0 \in K(R|_{\overline{R^{-1} (V)}}) \subseteq V =\widehat{\mathbb{C}} \smallsetminus \overline{D}$, so $y_0 \notin \overline{D}$. This contradicts the fact that $D \subseteq \mathcal{B} (R)$ is forward invariant, so $w' \in J(R)$. We have proved that in this case, $w' =\Phi (f( \eta (w)))$ is a parabolic periodic point, and thus is a corner of $\overline{V}$.
    Recall $\Phi (S(\sigma)) =\nu (V)$, so %$f(\eta (w))$ is a corner of $\overline{U}$, i.e., 
    $f(\eta (w)) \in S(\sigma)$. As a result, $\eta (w) \in (f|_{\overline{\mathbb{D}}})^{-1} (S(\sigma)) =P_1$, and $w \in P_1$. But by assumption, $w \notin P_1$. Hence $w' \notin \nu (V) \cup R^{-1} (\nu (V))$.% cannot be corners or pinchings of $\overline{R^{-1} (V)}$, and thus $f(\eta (w))$ is not corners or pinchings of $\sigma^{-1} (\overline{U})$.

    Recall that $\Phi^{-1}$ conjugates $R$ to $\sigma$ on the closure of a neighborhood of $\widehat{\mathbb{C}} \smallsetminus \mathcal{B} (R)$ pinched at all points in $\nu (V) \cup R^{-1} (\nu (V))$, so the fact that $w' =\Phi (f (\eta (w)))$ is a parabolic periodic point of $R$ implies that $f(\eta (w))$ is a parabolic periodic point of $\sigma$, by the classification of periodic points of conformal maps. However, we have already shown that $f(\eta (w))$ is not a parabolic periodic point of $\sigma$. Therefore, \eqref{9u99e988888} holds, and thus $\Lambda_+ \cap \mathcal{PC}_{F^{-1}} \subseteq \Omega_+$.
\end{proof}
\begin{rem}\label{re:importantadd}
    Note that the argument above also shows that $\Omega_+$ is finite, since $\Omega_+\subseteq P_1\cup (\Phi\circ f\circ \eta)^{-1}(\Omega_R)$, where $\Omega_R$ denotes the set of parabolic periodic points of $R$.
\end{rem}
Next, we have the following easy but important lemma.
\begin{lemma}\label{le:fnonsing}
If $f \in \mathcal{M}$, then $\Lambda_+\cap \sing=\emptyset$.
\end{lemma}
\begin{proof}
    Take $z\in \Lambda_+$. By definition of $\sing$, we need to show that $F(z)$ contains $d$ distinct points. This follows from the fact that $F$ is invariantly inverse-like on $\Lambda_+$, that if $w\in \Lambda_+$ is a critical point of $f$, then $w\in P_1$ by condition~(\ref{it:defM3}) in Definition~\ref{def:finM}, and lastly that the branch $f$ of $F$ from $w\in P_1$ to $w$ is locally well-defined and unique, since $f''(w)\neq 0$.
\end{proof}
The relative hyperbolicity of $
F$ on $\Lambda_+(x)$ is now immediate.
\begin{prop}\label{Frelativelyhyperbolic}
    If $f \in \mathcal{M}$, and $x$ is such that $\Lambda_+(x)=\Lambda_+$, then $F$ is relatively hyperbolic on $\Lambda_+(x)$.
\end{prop}

\begin{proof}
    As $\mathcal{CV}_{F^{-1}}\subseteq \sing$, this proposition follows from Lemma~\ref{le:fnonsing}, and Propositions~\ref{locally attracting} and~\ref{Lambda+capPCF-1=Omega}.\end{proof}

\begin{prop}\label{prop:LambdaMinimal}
    If $f \in \mathcal{M}$, and $x$ is such that $\Lambda_+(x)=\Lambda_+$, then $\Lambda_+(x)$ is minimal.
\end{prop}

\begin{proof}
    First, from Lemma~\ref{le:fnonsing}, we have that $\Lambda_+\cap(\sing\smallsetminus \mathcal{CV}_{F^{-1}})=\emptyset$.
    Next, we notice that $\Lambda_+$ is the boundary of the closed set $K_+$, so it has empty interior. We suppose $A$ is a nonempty closed subset of $\Lambda_+$ satisfying $F(A) \subseteq A$. Recall that $F(\Lambda_+) =\Lambda_+$ from \eqref{F^-1(K-)=K-F^-1(Lambda-)=Lambda-,F(K+)=K+,F(Lambda+)=Lambda+}. This implies $F^{-1} (z) \cap \Lambda_+ \neq \emptyset$ for each $z \in \Lambda_+$. For $z \in \Lambda_+$, recall $\{ g(z) \} =F^{-1} (z) \cap \overline{\mathbb{D}^*}$, so $\{ g(z) \} =F^{-1} (z) \cap \Lambda_+$. Hence, we obtain $F(z) =g^{-1} (z)$ for all $z \in \Lambda_+$, so $g^{-1} (A) \subseteq A$. By \eqref{sigmafetaz=fetagz}, we get $\sigma^{-1} ( (f\circ \eta ) (A)) \subseteq (f\circ \eta) (A)$. By Definition~\ref{def:finM}, there exists a homeomorphism $\Phi$ from $\partial K(\sigma)$ to $J(R)$ conjugating $\sigma$ to a Bers-like antirational map $R$ with no Fatou component being a Siegel disk or a Herman ring. Then $R^{-1} ((\Phi \circ f \circ \eta) (A)) \subseteq (\Phi \circ f \circ \eta) (A)$. Since $A$ is a nonempty closed subset of $\Lambda_+$, $f \circ \eta$ restricted on $\Lambda_+$ is a homeomorphism from $\Lambda_+$ to $\partial K(\sigma)$, and since $\Phi$ is a homeomorphism from $\partial K(\sigma)$ to $J(R)$, $(\Phi \circ f \circ \eta) (A)$ is a closed subset of $J(R)$ and we have $A =\Lambda_+$ if and only if $(\Phi \circ f \circ \eta) (A) =J(R)$. Since $A$ is nonempty and since $\bigcup_{n=0}^{+\infty} R^{-n} (z)$ is dense in $J(R)$, $R^{-1} ((\Phi \circ f \circ \eta) (A)) \subseteq (\Phi \circ f \circ \eta) (A)$ implies $(\Phi \circ f \circ \eta) (A) =J(R)$, so $A =\Lambda_+$. Hence $\Lambda_+$ has no proper subset that is $F$-invariant, and condition (\ref{it:min1}) in Definition~\ref{def:minimal} holds.

    %Now we suppose that there are two sequences $\{ z_j \}_{j=0}^n$ and $\{ z_j ' \}_{j=0}^n$ of points in $\widehat{\mathbb{C}}$ satisfying $z_j \in F(z_{j-1})$ and $z_j '\in F(z_{j-1} ')$ for all $j \in \{ 1 ,\, \dots ,\, n \}$, $z_0 =z_0 ' \in \Lambda_+$, and $z_n =z_n '\in \Lambda_+$.
    
    %As $F(\Lambda_+)=\Lambda_+$, we have $z_j \in \Lambda_+$ for all $j \in \{ 0 ,\, 1 ,\, \dots ,\, n \}$.
    %Indeed, we assume $z_{j+1} \in \Lambda_+$ and $z_j \notin \Lambda_+$ for some $j$. By \eqref{F^-1zsubseteqg(z)cupLambda-} we have $z_j \in F^{-1} (z_{j+1}) \smallsetminus \Lambda_+ \subseteq \Lambda_- \smallsetminus P_1 \subseteq \mathbb{D}$. Due to $F^{-1} (\mathbb{D}) \subseteq \mathbb{D}$, we have $z_j ,\, z_{j-1} ,\, \dots ,\, z_0 \in \mathbb{D}$ by induction. However, $z_0 \in \Lambda_+$, so $z_0 \notin \mathbb{D}$. Hence the claim holds. 
    %If $z_j =z_j '$ for some $j \in \{ 1 ,\, \dots ,\, n \}$, then $z_{j-1} =g(z_j) =z_{j-1} '$ because $g(z_j)$ is the only point in $\Lambda_+$ with $g(z_j) \in F^{-1} (z_j)$. By induction on $j$ we conclude $z_j =z_j '$ for all $j \in \{ 0 ,\, 1 ,\, \dots ,\, n \}$. Recall from condition~(\ref{it:rel1}) in Definition~\ref{def:wellbehaved} and Proposition~\ref{Frelativelyhyperbolic} that $F$ is non-branched on $\Lambda_+$. Thus, the argument above implies that for each $z,\, w \in \Lambda_+$ and $n \in \bN$, the multiplicity of $w$ in $F^n (z)$ is at most 1. Hence condition~(\ref{it:uniquebranch}) in Definition~\ref{def:minimal} holds.

    Now we aim to prove condition~(\ref{it:U}) in Definition~\ref{def:minimal}.
    To that end, fix an arbitrary $z \in \Lambda_+ \smallsetminus \Omega_+$. Recall from the proof of Proposition~\ref{Lambda+capPCF-1=Omega} that $\Lambda_+ \smallsetminus P_1 \subseteq \operatorname{int} (W)$, so by Proposition~\ref{Omega=P_1}, we have $z \in \Lambda_+ \smallsetminus \Omega_+ \subseteq \Lambda_+ \smallsetminus P_1 \subseteq \operatorname{int} (W)$.

    Suppose that there exists an attracting periodic point $y \in K_+ \cap \operatorname{int} (W) =K_+ \smallsetminus P_1$ of $g$. Then $g^n (y) \in K_+ \smallsetminus P_1 \subseteq \operatorname{int} (W)$ for all $n \in \bN$ (note that $g^n (y) \notin P_1$ is because $g(P_1) =P_1$ and $g^m (g^n (y)) =y \notin P_1$ for some $m \in \bN$). Then the fact that $y$ is attracting implies that there exists a neighborhood of $y$ where $g^n$ is defined for all $n \in \bN$, so by Lemma~\ref{pandingK+}~(1), we get $y \in \operatorname{int} (K_+)$.
    
    We choose $U_z$ to be an open neighborhood of $z$ in $\operatorname{int} (W)$ not containing any attracting periodic point of $g$ in $K_+ \smallsetminus P_1$.
    For an arbitrary $w \in F(W)$, there exists $z \in W \subseteq \mathbb{D}^* \cup P_1$ such that $w \in F(z)$, i.e., $z \in F^{-1} (w)$. By (\ref{it:le641}) in Lemma~\ref{pandingK+}, this implies $w \in W$. Thus, $F(W) \subseteq W$ and $g (z) \in W$ for all $z \in F(w)$. As a result, we get that $g^n$ is defined for all $z \in F^n (U_z) \subseteq F^n (W)$. By the definition of $g$, we can see $g^n (F^n (w)) =\{ w \}$ for all $w \in U_z$. Moreover, $F^n(U_z)\subseteq {\mathbb D^*}$, since $U_z\subseteq \mathbb D^*$, so there are three points in $\setRS$ that do not belong to $\bigcup_{n=0}^{+\infty} F^n(U_z)$.

    The statement~(\ref{it:le642}) in Lemma~\ref{pandingK+} implies that $K_+ =\bigcap_{n=0}^{+\infty} g^{-n} (W) =\bigcap_{n=0}^{+\infty} F^n (W)$. Since $F^n (W) =g^{-n} (W)$ is compact for all $n \in \bN$, and $\dots \subseteq F^3 (W) \subseteq F^2 (W) \subseteq F(W) \subseteq W$, we get that for each open neighborhood $\mathcal{U}$ of $K_+$, there exists $N \in \bN$ such that $F^N (W) \subseteq \mathcal{U}$. Consequently, the limit set of $\bigcup_{n=0}^{+\infty} F^n (y)$ is contained in $K_+$ for all $y \in U_z$.

    Now we fix an arbitrary $y \in U_z$ and claim that the limit set of $\bigcup_{n=0}^{+\infty} F^n (y)$ is contained in $\Lambda_+$. We will discuss two cases: $y \in K_+$ and $y \notin K_+$.

    Suppose $y \notin K_+$. By (\ref{it:le642}) Lemma~\ref{pandingK+}, there exists $m \in \bN$ with $g^m (y) \notin W$. For each $n \in \bN$ and each $w \in F^n (y)$, we have $g^n (w) =y$, and thus $g^{m+n} (w) \notin W$. This implies that $w \notin K_+$. Hence $\bigcup_{n=0}^{+\infty} F^n (y) \cap K_+ =\emptyset$, so its limit set is contained in $\widehat{\mathbb{C}} \smallsetminus \operatorname{int} K_+$. Hence, the limit set of $\bigcup_{n=0}^{+\infty} F^n (y)$ is contained in $(\widehat{\mathbb{C}} \smallsetminus \operatorname{int} K_+) \cap K_+ =\Lambda_+$.

    Now we suppose $y \in K_+$. By the choice of $U_z$, $y$ is not an attracting periodic point of $g$. Recall that $\Phi \circ f \circ \eta$ from $K_+$ to $\widehat{\mathbb{C}} \smallsetminus \mathcal{B} (R)$ conjugates $g$ to $R$ and that $F =g^{-1}$ in $K_+$. Then $(\Phi \circ f \circ \eta) (y)$ is not an attracting periodic point of $R$. The limit set of $\bigcup_{n=0}^{+\infty} F^n (y)$ is mapped to the limit set of $\bigcup_{n=0}^{+\infty} R^{-n} ((\Phi \circ f \circ \eta) (y))$ by $\Phi \circ f \circ \eta$.  Since $R$ has no Fatou component being a Siegel disk or a Herman ring, the limit set of $\bigcup_{n=0}^{+\infty} R^{-n} ((\Phi \circ f \circ \eta) (y))$ must be contained in $J(R)$, so the limit set of $\bigcup_{n=0}^{+\infty} F^n (y)$ is contained in $\Lambda_+$.

    Therefore, the limit set of $\bigcup_{n=0}^{+\infty} F^n (y)$ is contained in $\Lambda_+$ for all $y \in U_z$ and condition~(\ref{it:U}) in Definition~\ref{def:minimal} holds. Using Remark~\ref{re:importantadd}, we now conclude that $\Lambda_+$ is minimal.
\end{proof}

We can now conclude the first main result of this paper.

\begin{proof}[Proof of Theorem~\ref{thm:misha}]
    First, $\Lambda_+$ is not totally disconnected as it is homeomorphic to $J(R)=\partial \mathcal B(R)$.
    The condition that $R$ has an attracting periodic point in $\setRS\smallsetminus \mathcal{B}(R)$ implies that $g$ has an attracting fixed point in the interior of $K_+$. The set $K_+$ is closed and is a proper subset of $\setRS$, as it is disjoint from $\mathbb{D}$. Theorem~\ref{thm:misha} is thus a consequence of Lemmas~\ref{le:islimitset}, \ref {le:fnonsing}, Theorem~\ref{thm:existencemmeasure}, Propositions \ref{Frelativelyhyperbolic}, and \ref{prop:LambdaMinimal}.
\end{proof}

\subsection{Limit sets of Bullett--Penrose correspondences}\label{sec:bullett}
In this subsection, we study the limit sets of the Bullett--Penrose correspondences.
We briefly introduced the relevant concepts for Lemma~\ref{le:Bullettinverseetc} in Section~\ref{sec:introduction}, but for the full definitions we refer the reader to \cite{bullett2017mating}. We shall show the following lemma.
\begin{lemma}\label{le:Bullettinverseetc}
    Let $a\in \setRS$ belong to the interior of a hyperbolic component of the modular Mandelbrot set. Then $F_a$ is invariantly inverse-like on $\Lambda_{a,+}$ and on $\partial \Lambda_{a,+}$. Furthermore, for all except finitely many $x\in \Lambda_{a,+}$, $F_a$ is relatively hyperbolic on $\Lambda_+(x) =\partial \Lambda_{a,+}$ and $\Lambda_+(x)$ is minimal.
\end{lemma}

\begin{proof}First, by interchanging the roles of $F_a$ and $F_a^{-1}$ and the roles of $\Lambda_{a,+}$ and $\Lambda_{a,-}$, there exists by \cite[Main Theorem]{bullett2020mating} and its proof, a parameter $A$ in the interior of a hyperbolic component of the \emph{parabolic Mandelbrot set} such that $F_a^{-1}$ is a mating between the rational map $P_A=z+1/z+A$ and the modular group. See \cite[Introduction]{bullett2020mating} for a definition of the parabolic Mandelbrot set. 
In particular, the 2-to-1 branch $f_a$ of $F_a^{-1}$ under which $\Lambda_{a,+}$ is invariant is hybrid equivalent to $P_A$, see \cite[Definition 3.3]{LOMONACO_2015} for the definition of hybrid equivalence. Let $\Phi$ be a conjugation such that $P_A=\Phi^{-1}\circ f_a\circ \Phi$ on a pinched neighborhood $U_{K(P_A)}$ of the filled Julia set $K(P_A)$ of $P_A$, pinched at the parabolic fixed point $\infty$ and its other preimage, to a pinched neighborhood $\Phi(U_{K(P_A)})$ of $\Lambda_{a,+}(x)$, pinched at $F_a(0)$. Note that $\Phi(\infty)=0$.

The map $P_A$ has degree 2, hence has 2 critical points, $c_0$ and $c_1$, and has a unique parabolic fixed point and a unique attracting periodic orbit. One of the critical points of $P_A$ belong to the basin of attraction of the attracting periodic orbit and the other to the basin of attraction of the parabolic fixed point. In particular, $P_A$ has no critical points on its Julia set. Moreover, as 0 is a parabolic fixed point of $F_a$ and $F_a(z)=\Phi\circ P_A^{-1}\circ \Phi^{-1}$ on $\Phi(U_{K(P_A)})$, it follows that $0$ is the only parabolic periodic point of $F_a$ contained in $\Lambda_{a,+}$. Since $P_A^{-n}(x)\to J(P_A)$ as $n\to +\infty$ for all $x$ not belonging to the attracting periodic orbit of $P_A$, we have that $\Lambda(x)=\partial\Lambda_{a,+}$ for all $x\in \Lambda_{a,+}$ except those belonging to the image under $\Phi$ of the attracting periodic orbit of $P_A$. Let us fix an $x\in \Lambda_{a,+}$ such that $\Lambda_+(x)=\partial \Lambda_{a,+}$. Recall that the set of points belonging to an attracting periodic orbit of a rational map is finite.

Let us first show that $F_a$ is relatively hyperbolic on $\Lambda_+(x)$. We have mentioned that $\Omega_+=\{0\}$.
Since $P_A$ has no critical points on $J(P_A)$, each point in $\Lambda_+(x)$ has two distinct images under $F_a$. As $F_a$ is $2$-to-$2$, this implies that $\Lambda_+(x)\cap \sing=\emptyset$. Thus, condition~(\ref{it:rel1}) in Definition~\ref{def:wellbehaved} holds.

%The basin of the parabolic fixed point and the basin of the attracting periodic orbit contain at least one critical point of $P_A$ so the 2-to-1 branch $f_a$ of $F_a^{-1}$ for which $\Lambda_{a,+}$ is invariant has no critical point on $\partial \Lambda_{a,+}\smallsetminus F(0)$.
%We have that 0 is a parabolic fixed point of $F_a$ and $F_a(0)\smallsetminus \{0\}$ is a branch point of $F_a^{-1}$, see  \cite[Proposition~5.1]{bullett2017mating}. This implies that $\Lambda_+(x)\cap \mathcal{CV}_{F^{-1}}=\emptyset$ so 
%This argument also shows that $F_a$ can have no parabolic periodic points other than 0 in $\partial\Lambda_{a,+}$, as this would be a parabolic periodic point of $P_A$, but there exists no critical point of $P_A$ that can belong to the basin of said parabolic orbit, as $P_A$ only has 2 critical points. That is, $\Omega_+(x)=\{0\}$. Moreover, it follows that $\partial \Lambda_{a,+}\cap \operatorname{Sing}_{F_a}$, as the image under $F_a$ of each point $z$ contains two points, as $0$ contains two distinct images and $P_A$ contains no critical point in the Julia set $J(P_A)$ of $P_A$.
 
Now, following the proof of \cite[Theorem~B]{bullett2017mating} it is quite straightforward to see that condition~(\ref{it:attracting}) in Definition~\ref{def:wellbehaved} holds, see also \cite[Figure~6]{bullett2017mating}. Next, the complement of $\Lambda_{a,+}\cup \Lambda_{a,-}$ contains no critical points of $F_a^{-1}$, since $F_a^{-1}$ is conjugated to a pair of M\"obius transformations there. As $\Lambda_{a,+}\cup \Lambda_{a,-}$ is completely invariant, the fact that  $\Lambda_{a,-}$ is backward-invariant, and $\Lambda_{a,+}\cap \Lambda_{a,-}=\{0\}$, condition~(\ref{it:thirdrel}) in Definition~\ref{def:wellbehaved} follows. We have thus concluded that $F_a$ is relatively hyperbolic on $\Lambda_{a,+}$.

We turn to minimality of $\partial\Lambda_{a,-}$. We have already shown that $\Lambda_{a,+}\cap\sing=\emptyset$. Condition~(\ref{it:min1}) then holds since,  by the existence of the conjugation between $F_a$ and the inverse of $P_A$, $\partial\Lambda_{a,+}$ has no proper closed subsets that are forward invariant under $F_a$, and $\partial \Lambda_{a,-}$ is the boundary of a closed set, and hence has empty interior. 

Next, $\Phi(U_{K(P_A)})$ is a pinched neighborhood of $\Lambda_{+,a}$, pinched at $F_a(0)$ such that $F_a(z)=\Phi\circ P_A^{-1}\circ \Phi^{-1}(z)$ and by definition of $\Lambda_{a,+}$, we may choose $U_{K(P_A)}$ such that $F(\Phi(U_{K(P_A)}))\subseteq \Phi(U_{K(P_A)})$, see also \cite[Propsition 5.2  and Figure 6]{bullett2017mating}. Denote by $y$ the point in $F_a(0)$ not equal to $0$. Then we can, by the dynamics of $F_a$ near $y$, find an $\epsilon>0$ such that $F_a(\overline {B(y,\epsilon)})\subseteq \Phi(U_{K(P_A)})$. Then $U=\overline{B(y,\epsilon)}\cup \Phi(U_{K(P_A)})$ satisfies $F(U)\subseteq U\subseteq \Phi(U_{K(P_A)})$, $\Phi\circ P_A\circ \Phi^{-1}\circ F_a(z)=\{z\}$ for each $z\in U_{K(P_A)}$ and $U_{K(P_A)}\neq \setRS$. In order to conclude condition~(\ref{it:U}) of Defintion~\ref{def:minimal}, it suffices by Remark~\ref{remark:suff} to show that $F_a^n(z)\to \partial \Lambda_{a,+}$, as $n\to +\infty$, for each $z\in U$ except maybe a finite set of $z$. Now, taking $z= \Phi(w)$, where $w$ does not belong to an attracting periodic orbit of $P_A$, $F_a^n(z)\to \partial \Lambda_{a,+}$ as $n\to +\infty$ again follows from the fact that $P_A$ is a rational map with no critical points on its Julia set, since then $P_A^{-n}(w)\to J(P_A)$ as $n\to +\infty$. Since $\Omega_+=\{0\}$, minimality of $\partial \Lambda_{+,a}$ follows. 
\end{proof}

We are now in position to conclude the second main result of this paper. 
 %\begin{cor}
  %       Let $a$ belong to the interior of a hyperbolic component of the modular Mandelbrot set. Then there exists a non-atomic conformal measure on $\partial \Lambda_{a,+}$ and $1\leq \operatorname{HD}(\partial \Lambda_{a,+})<2$.
 %\end{cor}
 \begin{proof}[Proof of Theorem~\ref{thm:BP}]
     The condition that $a$ belongs to the interior of a hyperbolic component implies that the branch $f_a$ of $F_a^{-1}$ under which $\Lambda_{a,+}$ is invariant has an attracting periodic point in the interior of $\Lambda_{a,+}$. The theorem is now a direct consequence of Theorem~\ref{thm:existencemmeasure} and Lemma~\ref{le:Bullettinverseetc} together with the fact that $\partial \Lambda_{a,+}$ is a connected set that is not a singleton.
 \end{proof}
 In the same way, the same results hold for $F^{-1}_a$ and $\partial \Lambda_{a,+}$. We thereby deduce the following corollary.

\begin{cor}
    $1\leq \operatorname{HD}(\partial (\Lambda_{a,+}\cup \Lambda_{a,-}))<2$. As a consequence, $\partial (\Lambda_{a,+}\cup \Lambda_{a,-})$ has zero area.
\end{cor}

\section{Constructing conformal measures}\label{sec:2}
In this section, we construct the conformal measures using the Patterson--Sullivan construction. Let $x\in \setRS$. Recall the definition of the Poincar\'e series $P_s(x)$ \eqref{eq:poincare} and the critical exponent $\deltacrit(x)$ Definition~\ref{def:crit}. We have the following important proposition.
%As stated in Sections~\ref{sec:introduction} and \ref{sec:importantdefs}, we impose the condition  that $\Lambda_+(x)\cap \sing=\emptyset$, so that all branches of $F$ are defined in a neighborhood of every point $z\in \Lambda_+(x)$ and that $Df_i(z)<+\infty$ for all branches $f_i$ of $F$ defined in a neighborhood of $z\in \Lambda_+(x)$.

\begin{prop}\label{prop:fund}Let $F$ be an (anti)holomorphic correspondence that is invariantly inverse-like on a closed set $S$. Suppose that $x\in S$ is such that $0<\delta_{\operatorname{crit}}(x)<+\infty$ and $\Lambda_+(x)\cap \sing=\emptyset$. Then there exists a $\delta_{\operatorname{crit}}(x)$-conformal measure on $\Lambda_+(x)$.
\end{prop}
\begin{proof}
The arguments here are closely related to those found in e.g., the proof of \cite[Theorem~4.1]{mcmullen}, which is a similar statement but for the rational setting.
Recall from Remark~\ref{remark:invlike} that $\Lambda_+(x)\subseteq S$ and that $F$ is invariantly inverse-like on $\Lambda_+(x)$.

%Let $d_w$ be the degree of $P_{\min}(z,w)$ in $w$. 
%For each index $I=(i_1,i_2,\dots,i_n)$, with $1\leq i_j\leq d_w$ of length $n\geq 0$, $f_I$ denotes a unique branch of $F^n_{\min}$, so that $$\{f_I:I=(i_1,i_2,\dots,i_n), 1\leq i_j\leq d_w\}$$ contains all branches of $F^n_{\min}$. 
We shall handle the proof in the two separate cases, one in which $P_s(x)\to +\infty$ as $s\searrow \delta_{\operatorname{crit}}(x)$ and the other in which $P_s(x)\not\to \infty$ as $s\searrow \delta_{\operatorname{crit}}(x)$.

\smallskip \emph{Case 1.} Recall the definition of the Poincar\'e series $P_s(x)$ given in \eqref{eq:poincare}. Suppose that $P_s(x)\to +\infty$ as $s\searrow \delta_{\operatorname{crit}}(x)$. For each $s>\delta_{\operatorname{crit}}(x)$, define the measure
    \begin{equation*}
    	\mu_s\coloneqq \frac{1}{P_s(x)}     \sum_{n=0}^{+\infty} \sum_{j=1}^{M_n}|Df_{n,j}(x)|^s\delta_{f_{n,j}(x)},
    \end{equation*}
    where $\delta_{f_{n,j}(x)}$ is the Dirac mass at $f_{n,j}(x).$

    \smallskip \emph{Case 2}. Suppose now that $\lim_{s\searrow \delta_{\operatorname{crit}}(x)}P_s(x)$ does not diverge. We then alter the Poincar\'e series $P_s(x)$ slightly to ensure divergence. More specifically, define 
\begin{equation*}
    	t(n)\coloneqq\begin{cases}
        2\delta_{\operatorname{crit}}(x)-s& \text{ if }n< h (1/(s-\deltacrit(x)) ); \\
        s & \text{ otherwise, }
    \end{cases}
\end{equation*}
    where $h\:\bR\to \bR$ and $h(z)$ is sufficiently quickly increasing as $z\to +\infty$, depending on $F$ and $x$.
    We then define the new Poincar\'e series
    \begin{equation*}
    	P_s(x)\coloneqq  \sum_{n=0}^{+\infty} \sum_{j=1}^{M_n}|Df_{n,j}(x)|^{t(n)}
    \end{equation*}
    and the corresponding measures 
    \begin{equation*}
    	\mu_s\coloneqq \frac{1}{P_s(x)}     \sum_{n=0}^{+\infty} \sum_{j=1}^{M_n}|Df_{n,j}(x)|^{t(n)}\delta_{f_{n,j}(x)}.
    \end{equation*}

\smallskip We shall now show, regardless of whether we are in Case 1 or 2,  that each weak$^*$-limit, i.e., each weak$^*$-accumulation point, $\mu$ of the measures $\mu_s$ as $s\searrow \delta_{\operatorname{crit}}(x)$ is $\delta_{\operatorname{crit}}(x)$-conformal. We have that $\mu$ is supported on $\Lambda_+(x)$ since $P_s(x)\to +\infty$.

Regardless if we are in Case 1 or 2, take a special pair $(A,f)$ of $\Lambda_+(x)$, let $\mu$ be any weak$^*$-limit of the measures $\mu_s$ as $s\searrow \deltacrit(x)$ and let $\epsilon>0$ be given. As $\mu$ is supported on $\Lambda_+(x)$ and $(A,f)$ is a special pair, to conclude the proposition, we need to show that 
\begin{equation}\label{eqconffixed}
    \mu(f(A\cap \Lambda_+(x)))=\int_{A\cap \Lambda_+(x)}\!|Df|^{\deltacrit(x)}\,\mathrm{d}\mu.
\end{equation} 

Let us write
    \begin{align*}
    	m(\cdot)&\coloneqq\min\bigl\{|Df(\cdot)|^s, \, |Df(\cdot)|^{2\delta_{\operatorname{crit}}(x)-s}\bigr\}  \qquad \text{and}\\
    	M(\cdot)&\coloneqq\max\bigl\{|Df(\cdot)|^s, \, |Df(\cdot)|^{2\delta_{\operatorname{crit}}(x)-s}\bigr\}.
    \end{align*}
Since $\Lambda_+(x)\cap \sing=\emptyset$, we can find an open neighborhood $V$ of $\Lambda_+(x)$, and a real number $K\in (0,\infty)$ such that $|Df(y)|<K$ for all $y\in A\cap V$. This implies that $|Df|^s$, $m$, and $M$ converge uniformly on $V\cap A$ to $|Df|^{\deltacrit(x)}$ as $s\searrow \deltacrit(x)$, since $\deltacrit(x)>0$. Furthermore, we may choose $V$ such that $ \overline V\cap \sing =\emptyset$.
Recall that $\mathcal{CV}_{F^{-1}} \subseteq \sing$, so $\mathcal{CV}_{F^{-1}}\cap V=\mathcal{CV}_{F^{-1}}\cap \Lambda_+(x) =\emptyset$.

Suppose that $w_j\in f(\overline{A}\cap \Lambda_+(x))$ is such that there exists a sequence $\{z_n\}_{n\in \mathbb N}$ in $ A\cap V$ such that $|Df(z_n)|\to 0$, $f(z_n)\to w_j$, $z_n\to z_{0,j}$, and $z_{0,j}\in \overline{ A}\cap \Lambda_+(x)$. Then $w_j\in \mathcal{CV}_{F}$ and so there are finitely many such points $w_j$ and $z_{0,j}$. We denote by $W$ the set of all $w_j$ with the property above. For each $w_j\in W$, take an open neighborhood $U_{j}$ of $z_{0,j}$ such that all branches of $F$ are defined in $U_{j}$, and let $f_{w_j}$ be the unique branch of $F$ defined in $U_{j}$ such that $f_{w_j}(z_{0,j})=w_j$. Let the order of the critical point $z_{0,j}$ of $f_{w_j}$ be $d_{w_j}$. The set $U_{j}$ may be chosen so that if $f_{w_j}(z)=w$ and $z\neq z_{0,j}$, then $f_{w_j}^{-1}(w)$ contains $d_{w_j}$ points and $V$ may be chosen so that $F^{-1}(w)\cap V=f_{w_j}^{-1}(w)$. %\Xiaoran{Is $V$ fixed or to choose here?}
For each $w_j\in W$, we can further choose $U_{j}$ such that $\max \bigl \{|Df_{w_j}(y)|^{2\deltacrit(x)-s}, \, |Df_{w_j}(y)|^{s} \bigr\}<\epsilon/|W|$ for each $y\in U_{j}$, and each $s$ sufficiently close to $\deltacrit(x)$, where $|W|$ denotes the cardinality of $W$. 
This is, again, possible as $\deltacrit(x)>0$. We then have 
\begin{equation}\label{eq:trivialU}
    \sum_{\omega_j\in W}\int_{U_{j}}\!|Df_{w_j}|^{\deltacrit(x)} \,\mathrm{d}\mu<\epsilon.
\end{equation}

Note now that for each $w\in \Lambda_+(x)\smallsetminus \mathcal C_F,$
we can find a neighborhood $V_w$ of $w$ and a univalent branch $g$ of $F^{-1}$ defined on $V_w$ such that $g(w)=g_{F,\Lambda_+(x)}(w)$ and if $\zeta\in V_w\cap \Lambda_+(x)$, then $g(\zeta)=g_{F,\Lambda_+(x)}(\zeta)$. 
As $F$ is invariantly inverse-like on $S$ and on $\Lambda_+(x)$, $\Lambda_+(x)\cap \sing=\emptyset$, and $x\in S$, we can thereby choose $V$ above such that if $w\in f(z)\cap \bigcup_{n=1}^{+\infty}F^n(x)$ for some $z\in A\cap V$, but $A\cap V$ does not contain any point $y$ in $F^{-1}(w)\cap \bigcup_{n=0}^{+\infty}F^n(x)$ with $f(y)=w$, then there exists a unique $w_j\in W$ such that $y\in U_{j}$ and $f_{w_j}(y)=w$. Furthermore, there exists a unique $y_0\in U_{j}$ such that $y_0\in F^{-1}(w)\cap \bigcup_{n=0}^{+\infty}F^n(x)$, since $x\in S$ and $F$ is invariantly inverse-like on $S$.

Now, in order to simplify notation, write $\mathcal A\coloneqq f^{-1}\bigl(\bigcup_{w_j\in W} \bigl( f_{w_j}(U_{j}) \bigr)\bigr)$,  $ \mathcal B\coloneqq (A\cap V)\smallsetminus \mathcal A$, and $\mathcal C\coloneqq A\cap V \cap \mathcal A$.  Define also
\begin{equation*}
	G_s(B)\coloneqq 1 / P_s(x)
\end{equation*}
if %$y\in A$ is mapped to $x$ by $f$
$x \in B$, and $G_s(B)\coloneqq 0$ if %no such $y$ exists
$x \notin B$, for each $B\subseteq \setRS$. 
Then, by definition of $\mathcal C$, 
\begin{equation}\label{eq:name}
     f(\mathcal C)\subseteq \bigcup_{w_j\in W}f_{w_j}(U_{j}).
\end{equation}
By construction of $\mu_s$, definition of $f_{w_j}$, and the choice of $U_{j}$, regardless if we are in Case 1 or 2, we have

\begin{align*}
     \mu_s\Bigl(\bigcup_{w_j\in W}f_{w_j}(U_{j})\Bigr)< \epsilon + G_s\Bigl(\bigcup_{w_j\in W}f_{w_j}(U_{j})\Bigr),
\end{align*}
for all $s>\deltacrit(x)$ sufficiently close to $\deltacrit(x)$.
Since $G_s(B)\to 0$ as $s\searrow \deltacrit(x)$ for each $B\subseteq \setRS$, this shows that 
\begin{equation}\label{eq:leqepesilon}
    \mu(f(\mathcal C))\leq \epsilon
\end{equation} as $\bigcup_{w_j\in W}f_{w_j}(U_{j})$ is an open set.

Now, in Case 1, by construction of $\mu_s$, using the chain rule we have,
\begin{align*}
     \mu_s(f(  \mathcal B))= \int_{  \mathcal B} \!|Df|^s \,\mathrm{d}\mu_s+G_s(f(\mathcal B)).
\end{align*}
 It follows that as $s\searrow \deltacrit(x)$, since $f$ is locally injective on $\mathcal B$ and $\mathcal B\subseteq A\cap V$, we have
\begin{equation}\label{eq:eqbla}
     \mu(f(  \mathcal B))= \int_{  \mathcal B} \!|Df|^{\deltacrit(x)} \,\mathrm{d}\mu\leq \int_{A\cap V} \!|Df|^{\deltacrit(x)} \,\mathrm{d}\mu.
\end{equation}Then, using \eqref{eq:name} and \eqref{eq:leqepesilon}, we have
\begin{equation}\label{eq:confineq1}
    \mu(f(A\cap V))=  \mu(f(  \mathcal B))+\mu(f(\mathcal C))\leq \int_{A\cap V} \!|Df|^{\deltacrit(x)} \,\mathrm{d}\mu+\epsilon.
\end{equation}
On the other hand, using \eqref{eq:trivialU}, \eqref{eq:eqbla}, and that $A\cap V=\mathcal B\sqcup \mathcal C$, we have
\begin{equation}\label{eq:confineq2}
\mu(f(A\cap V))
\geq \mu(f(  \mathcal B))=\int_{  \mathcal B}\!|Df|^{\deltacrit(x)}\,\mathrm{d}\mu
> \int_{A\cap V}\!|Df|^{\deltacrit(x)}\,\mathrm{d}\mu-\epsilon. 
\end{equation}
As $\epsilon$ was arbitrary, $(A,f)$ is a special pair of $\Lambda_+(x)$, and $\mu$ has support contained in $\Lambda_+(x)$, \eqref{eqconffixed} follows. 

In Case 2, by construction of $\mu_s$ and the chain rule, we have \begin{align*}
        \int_{  \mathcal B} \!m(\cdot)\,\mathrm{d}\mu_s+G_s(   f(\mathcal B))\leq \mu_s(f(   \mathcal B))
        \leq \int_{  \mathcal B} \!M(\cdot)\,\mathrm{d}\mu_s+G_s(   f(\mathcal B)).
    \end{align*}
Since $G_s(   \mathcal B)\to 0$ as $s\searrow \deltacrit(x)$, and  $m$, and $M$ converge uniformly on $V\cap A$ to $|Df|^{\deltacrit(x)}$ as $s\searrow \deltacrit(x)$, it follows that when sending $s\searrow \deltacrit(x)$, each weak$^*$-limit $\mu$ of the measures $\mu_s$ as $s\searrow\delta$ satisfies

\begin{equation*}
	\mu(f(   \mathcal B))= \int_{   \mathcal B} \!|Df|^{\deltacrit(x)} \,\mathrm{d}\mu.
\end{equation*}
In the same way as in Case 1, using \eqref{eq:trivialU} and \eqref{eq:name}, we obtain \eqref{eq:confineq1} and \eqref{eq:confineq2}, which imply \eqref{eqconffixed}. See also \cite[Theorem~3]{sullivan} for an alternative reference.
\end{proof}

\begin{definition}
    Any weak$^*$-limit of the measures $\mu_s$ as $s\searrow \deltacrit(x)$ constructed in Proposition~\ref{prop:fund} will be called a \emph{limit measure}.
\end{definition}

Take now a parabolic periodic point $\omega$ of $F$ and suppose that $F$ is invariantly inverse-like on $\Lambda_+(x)$ and that $\Lambda_+(x)\cap\sing=\emptyset$.
Let $q(\omega)\geq 1$ be the minimal natural number such that $\omega\in F^{q(\omega)}(\omega)$ and denote by $T_\omega$ the (unique) locally defined branch of $F^{q(\omega)}$ fixing $\omega$. Note that $\omega$ is a parabolic periodic point of $F$ if and only if there exists a natural number $n\geq 1$ such that $g_{F,\Lambda_+(x)}^n(\omega)=\omega$ and $Dg_{F,\Lambda_+(x)}^n(\omega)$ is a root of unity. Under the assumption that $F$ is locally $\Omega_+(x)$-attracting on $\Lambda_+(x)$, there exist integers $n\geq 1$ and $p(\omega)\geq 1$, a complex number $a\neq 0$, and a neighborhood $V_{\omega}$ of $\omega$ where $T_\omega^m$ is defined for $m=1,\,\dots,\,n$ such that for each $z\in V_\omega$,
\begin{align*}
	T_{\omega}^n(z)&=\omega + ({z}-\omega)+a( z-\omega)^{p(\omega)+1}+\cdots \quad \text{or}\\
	T_\omega^n(z) &=\omega + (\overline{z}-\overline\omega)+a(\overline z-\overline\omega)^{p(\omega)+1}+\cdots.
\end{align*}

For each $z\in V_\omega$, provided $V_\omega $ is chosen sufficiently small, there exists a constant $C(z)\geq 1$ such that
\begin{equation*}
	n^{(p(\omega)+1)/p(\omega)} C(z) 
	\geq 1 / |DT_{\omega}^n(z)| 
	\geq  n^{(p(\omega)+1)/p(\omega)} / C(z),
\end{equation*} 
provided that $T^n_\omega(z)\in V_\omega$ for all $n\in \bN$, and that this implies $T_\omega^n(z)\to \omega$ as $n\to +\infty$, see \cite[Theorem~8.4]{urbanskidenker6}.

We will need the K\"obe distortion theorem and formulate it in the following way:

\begin{theorem}[K\"obe distortion theorem]\label{thm:kobe}
There exists an increasing function $k\:[0,1)\to [1,+\infty)$ such that each real number $r>0$ and each univalent holomorphic or antiholomorphic function $f\:B(z,r)\to \setRS$, we have 
\begin{equation*}
 |Df(y)| / |Df(y')| \leq k(t)
\end{equation*}
if $t\in [0,1)$, $|y-z|\leq tr$, and $|y'-z|\leq tr.$
\end{theorem}

We can now show the following proposition.

\begin{prop}\label{le:nonatomic} Suppose that $F$ is invariantly inverse-like on a closed set $S\subseteq \setRS$, $x\in S$, $\omega\in \Omega_+(x)$, and $\Lambda_+(x)\cap \sing=\emptyset$. Suppose further that 
$F$ is locally $\Omega_+(x)$-attracting on $\Lambda_+(x)$ and that the critical exponent $\delta_{\operatorname{crit}}(x)$ satisfies
\begin{equation*}
	p(\omega)/(p(\omega)+1)<\delta_{\operatorname{crit}}(x)<+\infty.
\end{equation*}
Then each limit measure has no mass on $\omega$, on attracting or repelling periodic points, or on their orbits. In particular, if
\begin{equation*}
	\sup_{\omega\in \Omega_+(x)} \{p(\omega)/(p(\omega)+1)\}<\delta_{\operatorname{crit}}(x)<+\infty,
\end{equation*}
then each limit measure does not have mass on $\Omega_+(x)$. 
\end{prop}
\begin{proof}
    We use the construction from Proposition~\ref{prop:fund} and show that each limit measure has no mass on any parabolic periodic point of $F$.

    Fix now $\omega\in \Omega_+(x)$. 
    As $\delta_{\operatorname{crit}}(x)<+\infty$, it follows that $\omega\notin \bigcup_{n=0}^{+\infty} F^n(x)$. Indeed, $|DT_\omega^n(\omega)|=1$ for each $n\geq 1$, so the Poincar\'e series of any point whose forward orbit lands on $\omega$ diverges for each $s>0$.

    Now, in a similar way as in \cite[Theorem~8.7]{urbanskidenker6} or \cite[Theorem~4.1]{mcmullen}, by the conditions of $\Lambda_+(x)$ and compactness, together with the local dynamics of an analytic germ around a rationally indifferent fixed point, we may find finitely many points $x_i$, $i \in \{1, \, \dots, \, N \}$, and corresponding balls $B(x_i,r_i)$ such that the branch $T_\omega^n$ of $F^{nq(\omega)}$ fixing $\omega$ is defined and injective on $B(x_i,2r_i)\subseteq V_\omega$ for each $n\geq 0$. Furthermore, defining $\mathcal{S}\coloneqq \bigcup_{i=1}^N \bigcup_{n=0}^{+\infty} T_\omega^n (B(x_i,r_i)),$ we can choose these balls such that
    %$\mathcal{S}$ covers $\Lambda_+(x)$ sufficiently close to $\omega$ (not including $\omega)$ and by picking them sufficiently close to $\omega$,
    if the sequence $\{z_m\}_{m\in \bN}$ with $z_m\in F^{m}(x)$ has $\omega$ as a limit point, then there exists a natural number $m_0\geq 0$ such that $z_{m_0}\in B(x_i,r_i)$ for some $i\in \{1, \, \dots, \, N\}$ and $z_m\notin \mathcal S$ for each non-negative integer $m<m_0$. Note here that we use that $F$ is invariantly inverse-like on $\Lambda_+(x)$ and continuity of $F$ (in the sense that $F(z)$ depends continuously on $z$ in
the Hausdorff metric). As an illustrating example, for the correspondence in Example~\ref{ex:example}, the points $x_i$ can be chosen to lie on a red circular arc as in Figure~\ref{fig:Uneighborhood}, provided this circular arc is chosen sufficiently close to $\omega=1/2$. 
    
    Take now a sufficiently small open neighborhood $V$ of $\omega$, a real number $s>\deltacrit(x)$, and let $\epsilon>0$ be given. Notice by definition of the balls $B(x_i,r_i)$, $\mu_s(V)=\mu_s(V\cap \mathcal{S})$ for all $s$, provided that $V$ is sufficiently small. Let us for each integer $j\geq 1$, define
    \begin{equation*}
    	V_j\coloneqq V\cap \biggl(\bigcup_{i=1}^N \bigcup_{n=j}^{+\infty} T_\omega^n (B(x_i,r_i))\biggr).
    \end{equation*} 
    Then,
    \begin{align*}
        \mu_s(V_j)
    \leq \frac{1}{P_s(x)}\sum_{m=0}^{+\infty} \sum_{\substack{1\leq j\leq M_m\\f_{m,j}(x)\in V_j}}|Df_{m,j}(x)|^s  
    &\leq K^s\sum_{i=1}^N\sum_{n=j}^{+\infty} |DT_\omega^n(x_i)|^s\mu_s(B(x_i,r_i))  \\
    &\leq  N K^s C\sum_{n=j}^{+\infty} \frac{1}{n^{s(p(\omega)+1)/p(\omega)}}
    <\epsilon,
    \end{align*}
    provided that $j$ is sufficiently large, independently of $s$ when it is sufficiently close to $\delta_{\operatorname{crit}}(x)$, $K\coloneqq k(1/2)$, where $k\:[0,1)\to [1,+\infty)$ is the function from the K\"obe distortion theorem (Theorem~\ref{thm:kobe}), and
    \begin{equation*}
    	C\coloneqq \max\{C(x_i) : i\in \{1, \, \dots,  \, N \}\}.
    \end{equation*}
    From this, and that $\mu_s(\omega)=0$ 
 since $\delta_{\operatorname{crit}}(x)<+\infty$, it follows that if $\mu$ is a limit measure as $s\searrow \delta_{\operatorname{crit}}(x)$, then 
    $\mu(\omega)=0$.
    For an attracting or repelling periodic point $y$ of $F$, let $f_I$ be a branch of $F^q$ for some $q\geq 1$ fixing $y$ with $|Df_I(y)|\neq 1$. Then $\mu(y)=\mu(f^{n}(y))=|Df^n(y)|^{\delta} \mu(y)$  shows that $\mu(y)=0$. Finally, as $F$ is non-branched on $\Lambda_+(x)$, we have $Df(y)<K_0<+\infty$, and it follows that the bi-infinite orbit of each repelling or attracting periodic point of $F$ has measure 0, since $\mu$ is conformal.
\end{proof}

\section{Conformal measures for relatively hyperbolic correspondences on minimal forward limit sets}\label{sec:3}
In this section, we study conformal measures on forward limit sets $\Lambda_+(x)$ on which $F$ is relatively hyperbolic, and that are minimal. 

The objective of this section is to show the following proposition and Theorem~\ref{thm:main}.

\begin{prop}\label{prop:HDmain}  Suppose that $F$ is relatively hyperbolic on $\Lambda_+(x)$ and that $\Lambda_+(x)$ is minimal.
    If $\mu$ is a $\delta$-conformal measure that is not purely atomic, then $\delta<2$ and $\operatorname{HD}(\Lambda_+(x))\leq \delta$. 
\end{prop}

The arguments in this section are heavily influenced by and many times very similar to those used in \cite{urbanskidenker1}.
However, a significant difference between the rational setting and the one we study here is the possible existence of conformal measures whose support is contained in \emph{finitely} many points. For this reason, we introduce the notion of \emph{open} conformal measures, which by definition are conformal measures that are positive on each open subset of $\Lambda_+(x)$. With this definition in place, we use the ideas of the aforementioned paper to show results about open conformal measures, and in particular non-atomic conformal measures, including Proposition~\ref{prop:HDmain}. Applying this result, Lemma~\ref{le:atoms} below, and the results of Section~\ref{sec:2}, we subsequently conclude Theorem~\ref{thm:main}.

We now fix $x\in \setRS$, assume that $F$ is relatively hyperbolic on $\Lambda_+(x)$, and that $\Lambda_+(x)$ is minimal. The reader may note that it is not strictly necessary that the set $\Lambda_+(x)$ in the statements below is in fact a forward limit set for the arguments to hold. However, in order to apply the results of this section to Theorems~\ref{thm:misha}, \ref{thm:BP}, and~\ref{thm:main}, it is necessary. Therefore, we state the results in this way.

 The arguments below are tailored for the situation $\Omega_+(x)\neq \emptyset$. However, the same results hold in the case $\Omega_+(x) =\emptyset$ using the same arguments, but could be significantly simplified in many cases. 

Let us begin with some preliminary remarks about conformal measures in the present setting. 

First, note that $F(\Lambda_+(x))=\Lambda_+(x)$. Indeed, this follows from the fact that the image under $F$ of a compact set is again compact, which may be proven the same way one proves that the continuous image of a continuous map of a compact set is compact.

Recall from Section~\ref{sec:importantdefs} that we denote by $g_{F,\Lambda_+(x)}\:\Lambda_+(x)\to \Lambda_+(x)$ the map $g_{F,\Lambda_+(x)}(w)=z$, where $z$ is the unique point in $\Lambda_+(x)$ such that $w\in F(z)$. Note that
\begin{equation*}
Dg_{F,\Lambda_+(x)}(w)\coloneqq  1 / Df(z) ,
\end{equation*}
where $f$ is the unique branch of $F$ that maps $z$ to $w$, is well-defined. Here we use that $\Lambda_+(x)\cap \sing=\emptyset$. Observe that $|Dg_{F,\Lambda_+(x)}(w)|=+\infty$ if $Df(z)=0$. However, $g_{F,\Lambda_+(x)}$ has no critical points on $\Lambda_+(x)$, i.e., $Dg_{F,\Lambda_+(x)}$ does not vanish on $\Lambda_+(x)$. Hence, a measure $\mu$ supported on $\Lambda_+(x)$ is $\delta$-conformal for $F$ if

\begin{equation}\label{eq:gconf}
    \mu(g_{F,\Lambda_+(x)}(A))=\int_A\!|Dg_{F,\Lambda_+(x)}|^\delta \,\mathrm{d}\mu
\end{equation}
for each Borel set $A\subseteq \Lambda_{+}(x)$ on which $g_{F,\Lambda_+(x)}$ is injective. This can be shown for instance by using the chain rule of the Radon--Nikodym derivative. Furthermore, if $\mu$ is $\delta$-conformal for $F$ and $\mu$ has zero mass on $\mathcal{C}_F \cap \Lambda_+ (x)$, then \eqref{eq:gconf} holds for each Borel set $A\subseteq \Lambda_{+}(x)$ on which $g_{F,\Lambda_+(x)}$ is injective. Lastly, if $\mu$ is $\delta$-conformal for $F$, then \eqref{eq:gconf} holds for each Borel set $A\subseteq \Lambda_{+}(x)$ on which $g_{F,\Lambda_+(x)}$ is injective, and that does not contain the critical values $\mathcal{CV}_F$ of $F$.

Note that by condition~(\ref{it:U}) of minimality of the forward limit set, it follows that $\mathcal{C}_F\cap \Lambda_+(x)\subseteq \Omega_+(x)$.

Next, relative hyperbolicity and minimality lead to the following two useful lemmas.

\begin{lemma}\label{le:1} Suppose that $F$ is relatively hyperbolic on $\Lambda_+(x)$ and $\Lambda_+(x)$ is minimal.
Consider a real number $\Theta>0$. There exists a real number $\epsilon=\epsilon(\Theta)>0$ such that for all $z\in \Lambda_+(x)\smallsetminus B(\Omega_+(x),\Theta)$, the ball $B(z,2\epsilon)$ does not intersect $\mathcal{PC}_{F^{-1}}$. Consequently, for each $n\geq 0$, all branches of $F^n$ are defined on $B(z,2\epsilon)$. Furthermore, these branches defined on $B(z,2\epsilon)$ are injective.
\end{lemma}

\begin{proof}
    The statement that there exists a real number $\epsilon=\epsilon(\Theta)>0$ such that for all $z\in \Lambda_+(x)\smallsetminus B(\Omega_+(x),\Theta)$, the ball $B(z,2\epsilon)$ does not intersect $\mathcal{PC}_{F^{-1}}$ follows from condition~(\ref{it:thirdrel}) of relative hyperbolicity of $F$. What remains to prove is the injectivity of these branches. By condition~(\ref{it:U}) of minimality, for each $z\in \Lambda_+(x)\smallsetminus \Omega_+(x)$, there exists a neighborhood $U_z$ of $z$ such that for each $n\geq 0$, $g_{F,\Lambda_+(x)}^n$ extends to a map on $F^n(U_z)\cup \Lambda_+(x)$, still denoted by $g_{F,\Lambda_+(x)}^n$ (so that $g_{F,\Lambda_+(x)}^n(F^n(y))=\{y\}$ for each $y\in U_z$). By compactness of $\Lambda_+(x)\smallsetminus B(\Omega_+(x),\Theta)$, we can assume that $2\epsilon>0$ is smaller than some Lebesgue's number of the cover 
    \begin{equation*}
    \{\{U_z\}:z\in \Lambda_+(x)\smallsetminus B(\Omega_+(x),\Theta)\}.
    \end{equation*}
    Then $g_{F,\Lambda_+(x)}^n(f_n(y))=y$ for each $n\geq 1$, each $y\in B(z,2\epsilon)$, and each branch $f_n$ of $F^n$, where $g_{F,\Lambda_+(x)}^n$ is extended to $F^n(B(z,2\epsilon))\cup \Lambda_+(x)$. From this, injectivity follows and this concludes the proof of the lemma.
\end{proof}

    Let us define
    \begin{equation*}
    	X\coloneqq \Lambda_+(x)\smallsetminus \bigcup_{n=0}^{+\infty} F^n(\Omega_+(x))=  \Lambda_+(x)\smallsetminus \bigcup_{n=0}^{+\infty} g_{F,\Lambda_+(x)}^{-n}(\Omega_+(x)).
    \end{equation*}
    The set $X$ will be used throughout the rest of this section.

 \begin{lemma}\label{le:2} Suppose that $F$ is relatively hyperbolic on $\Lambda_+(x)$ and that $\Lambda_+(x)$ is minimal. Then there exists $\Theta>0$ such that  for each $z \in X$, we have $g_{F,\Lambda_+(x)}^n(z)\in \Lambda_+(x)\smallsetminus B(\Omega_+(x),\Theta)$ for infinitely many $n\in \bN$.
\end{lemma}

\begin{proof}
This  lemma follows from the fact that $F$ is locally $\Omega_+(x)$-attracting on $\Lambda_+(x)$ and that $\Omega_+(x)$ is finite. Indeed, we first pick $\Theta>0$ so that for each $\omega\in \Omega_+(x)$ of period $q$, we have $g_{F,\Lambda_+(x)}^{q}(z)=T^{-1}_{\omega}(z)$ for each $z\in B(\omega,\Theta)\cap \Lambda_+(x)$ (see Definition~\ref{def:Omega-attracting}). This is possible since $F$ is invariantly inverse-like on $\Lambda_+(x)$ and that $\Lambda_+(x)$ is closed. Then, if necessary, we decrease $\Theta$ so that 
\begin{equation*}
	 B(\omega_2,\Theta) \cap g_{F,\Lambda_{+}(x)}^{q}(B(\omega_1,\Theta)\cap \Lambda_+(x)) =\emptyset
\end{equation*}
for all distinct parabolic periodic points $\omega_1$ and $\omega_2$, where $\omega_1$ has period $q$. Then, we decrease $\Theta>0$ if necessary so that for each parabolic periodic point $\omega$ of period $q$, we have $g_{F,\Lambda_+(x)}^{q} (B(\omega,\Theta)\cap \Lambda_+(x) )\subseteq U$, where $U$ is the set appearing in Definition~\ref{def:Omega-attracting} corresponding to $\omega$. Then, for each $z\in \Lambda_+(x)\cap B(\Omega_+(x),\Theta)\smallsetminus\Omega $, there is $n\geq 0$ such that $g_{F,\Lambda_+(x)}^n(z)\notin B(\Omega_+(x),\Theta)$, by the properties of $T_\omega$ and $U$ in Definition~\ref{def:Omega-attracting}.
\end{proof}

Going forward, we fix $\Theta>0$ small enough for Lemma~\ref{le:2} to hold. Then, we choose $\epsilon>0$ so that Lemma~\ref{le:1} holds and such that if $y\in B(z,2\epsilon)$ for some $z\in \Lambda_+(x)\smallsetminus B(\Omega_+(x),\Theta)$ and $F(y)\cap \Lambda_+(x)\neq \emptyset$, then $y\in \Lambda_+(x)$. This is possible by Remark~\ref{re:1}. We keep these notations throughout this section.
%Moreover, for each $z\in \Lambda_+(x)\smallsetminus B(\Omega_+(x),\Theta)$ and each $n\geq 1$, we denote $g_{F,\Lambda_+(x),z}^n$ the map $g_{F,\Lambda_+(x)}^n$ extended to $F^n(B(z,2\epsilon))$.

\begin{definition}
    A conformal measure is \emph{open} if it is positive on all open nonempty subsets, in the subspace topology, of $\Lambda_+(x)$.
\end{definition} Since $\Lambda_+(x)$ does not contain any proper closed forward invariant subset, for each point $z\in \Lambda_+(x)$, the set $\bigcup_{n=0}^{+\infty} F^n(z)$ is dense in $\Lambda_+(x)$ (for otherwise $\overline{\bigcup_{n=0}^{+\infty} F^n(z)}$ would be a forward invariant closed proper subset).
We can now show the following important lemma.

\begin{lemma}\label{le:3}
	Suppose that $F$ is relatively hyperbolic on $\Lambda_+(x)$, that $\Lambda_+(x)$ is minimal, and that $\mu$ is an open $\delta$-conformal measure. Then there is a real number $B(\mu)\geq 1$ such that the following holds. 
	
	For each $z\in X$, there exists a sequence $\{r_j(z)\}_{j\in \bN}$ of positive numbers converging to 0 as $j\to +\infty$, such that  
\begin{equation*}
	\frac{1}{B(\mu)}\leq \frac{\mu\bigr(B(z,r_j(z))\bigr)}{r_j^{\delta}(z)}\leq B(\mu).
\end{equation*}
\end{lemma}

\begin{proof}
Let $z\in X$ be given and $\mu$ be a $\delta$-conformal measure on $\Lambda_+(x)$. Take $\{n_j\}_{j\in \bN}\coloneqq\{n_j(z)\}_{j\in \bN}$ a sequence such that $\lim_{j\to \infty} n_j= +\infty$ and $g_{F,\Lambda_+(x)}^{n_j}(z)\in \Lambda_+(x)\smallsetminus B(\Omega_+(x),\Theta)$ for all $j\in \bN$, which exists by Lemma~\ref{le:2}. By Lemma~\ref{le:1}, there exists for each $j$ a unique analytic branch $f_{n_j}$ of $F^{n_j}$ defined on $B\bigl(g_{F,\Lambda_+(x)}^{n_j}(z),2\epsilon\bigr)$ determined by the condition that $f_{n_j} \bigl( g_{F,\Lambda_+(x)}^{n_j}(z) \bigr)=z $ and this branch is univalent. 

We set $K\coloneqq k(1/2)$, where $k\:[0,1)\to [1,+\infty)$ is the function in the K\"obe distortion theorem (Theorem~\ref{thm:kobe}), and 
\begin{equation*}
	r_j=r_j(z)\coloneqq\frac{\epsilon|Df_{n_j} (g_{F,\Lambda_+(x)}^{n_j}(z))|}{K}=\frac{\epsilon}{K|Dg_{F,\Lambda_+(x)}^{n_j}(z)|}.
\end{equation*}
We then have that
\begin{equation}\label{eq:1}
	f_{n_j}\bigl(B \bigl(g_{F,\Lambda_+(x)}^{n_j}(z),\epsilon\bigr)\bigr)\supseteq B(z,r_j)
\end{equation}
and
\begin{equation}\label{eq:2}
	B\bigl( g_{F,\Lambda_+(x)}^{n_j}(z),\epsilon/K^2\bigr)\subseteq g_{F,\Lambda_+(x)}^{n_j}( B(z,r_j)),
\end{equation}
where for each $j$, $g_{F,\Lambda_+(x)}^{n_j}$ in \eqref{eq:2} is extended to $F^{n_j}(B(g^{n_j}(z),2\epsilon))\cup \Lambda_+(x)$. Throughout the rest of this proof, we shall assume that for each $j\in \mathbb N$, $g_{F,\Lambda_+(x)}^{n_j}$ has been extended in this way.
The inclusion \eqref{eq:1} gives that $g_{F,\Lambda_+(x)}^{n_j}$ is injective on $B(z,r_j)$. Note that this implies 
\begin{equation*}
\mu\bigl(g_{F,\Lambda_+(x)}^{n_j}(B(z,r_j))\bigr)
=\int_{ B(z,r_j)}\!\bigl|Dg_{F,\Lambda_+(x)}^{n_j}\bigr|^{\delta}\,\mathrm{d}\mu,
\end{equation*}
as for each $j\in \mathbb N$, $g_{F,\Lambda_+(x)}^{n}(B(z,r_j))$ does not intersect $\mathcal{CV}_{F^{-1}}$ for any $n=0,\,1,\,\dots,\, n_j-1$. Since $\Lambda_+(x)$ is compact and $\mu$ is open, 
\begin{equation*}
M\coloneqq \inf\{\mu(B(y,\epsilon/K^2)) : y\in \Lambda_+(x)\}>0.
\end{equation*}
By the K\"obe distortion theorem, \eqref{eq:1}, and \eqref{eq:2}, we find that 

\begin{align*}
    1&\geq \mu(g_{F,\Lambda_+(x)}^{n_j}(B(z,r_j)))=\int_{ B(z,r_j) }\!\bigl|Dg_{F,\Lambda_+(x)}^{n_j}\bigr|^{\delta}\,\mathrm{d}\mu\\
    &\geq K^{-\delta}\bigl|Dg_{F,\Lambda_+(x)}^{n_j}(z)\bigr|^{\delta}\mu\bigl(B(z,r_j)\bigr)
=\epsilon^{\delta}K^{-2\delta}r_{j}^{-\delta}\mu(B(z,r_j))  \qquad \text{and} \\
    M&\leq \mu\bigl(g_{F,\Lambda_+(x)}^{n_j}(B(z,r_j))\bigr)=\int_{  B(z,r_j)}\bigl|Dg_{F,\Lambda_+(x)}^{n_j}\!\bigr|^{\delta}\,\mathrm{d}\mu\\&\leq K^{\delta}\bigl|Dg_{F,\Lambda_+(x)}^{n_j}(z)\bigr|^{\delta}\mu(B(z,r_j)
=\epsilon^{\delta}r_j^{-\delta}\mu(B(z,r_j)).
\end{align*}

We will now show that $r_j\to 0$ as $j\to +\infty$. By compactness of $\Lambda_+(x)\smallsetminus B(\Omega_+(x),\Theta)$, after passing to a subsequence if necessary we can assume that $g_{F,\Lambda_+(x)}^{n_j}(z)\to y\in \Lambda_+(x)\smallsetminus B(\Omega_+(x),\Theta)$ as $j\to +\infty$ and that $g_{F,\Lambda_+(x)}^{n_j}(z)\in B(y,2\epsilon)$ for all $j\in \bN$. We can for each $j\in \bN$ find a unique branch $f_{n_j}$ of $F^{n_j}$ defined on $B(y,2\epsilon)$ such that $f_{n_j} \bigl( g_{F,\Lambda_+(x)}^{n_j}(z) \bigr)=z$. There are infinitely many $n_j$ that are odd or infinitely many $n_j$ that are even (or both). So we may suppose that all $n_j$ have the same parity. By condition~(\ref{it:U}) in Definition~\ref{def:minimal} and the fact that $\Lambda_+(x)$ has empty interior, this family of inverse branches is normal. Furthermore, all limit functions as $j\to +\infty$ are (anti)holomorphic (depending on the parity of $n_j$). Moreover, the limit functions are constants, since if not, then there would be a limit function $f$ such that $f(B(y,2\epsilon))$ contains a point $w$ in the complement of $\Lambda_+(x)$, because $f$ is (anti)holomorphic and non-constant, hence an open function and $\Lambda_+(x)$ has empty interior. This would imply that $w\in \Lambda_+(x_1)$, for some $x_1\in B(y,2\epsilon)$, contradicting condition~(\ref{it:U}) of Definition~\ref{def:minimal}, see also \cite[Theorem~6.4]{Brolin1965}.  Hence, after passing to a subsequence of $\{n_j\}_{j\in \bN}$ if necessary, we find that $\lim_{j\to +\infty} Df_{n_j} \bigl( g_{F,\Lambda_+(x)}^{n_j}(z) \bigr) \to 0$, which shows that
\begin{equation*}
	\bigl|Dg_{F,\Lambda_+(x)}^{n_j}(z)\bigr|= \bigl|Df_{n_j}\bigl(g_{F,\Lambda_+(x)}^{n_j}(z)\bigr)\bigr|^{-1}\to +\infty.
\end{equation*}
This in turn shows that $r_j\to 0$.
The lemma follows.
\end{proof}

\begin{prop}\label{prop:area}  
	Suppose that $F$ is relatively hyperbolic on $\Lambda_+(x)$ and that $\Lambda_+(x)$ is minimal. Then the area, i.e., the $2$-dimensional Lebesgue measure, of $\Lambda_+(x)$ is zero.
\end{prop}

\begin{proof}  
This proof follows the proof of \cite[Theorem~22.2]{Mishasbook}.
    By condition~(\ref{it:countable}) in Definition~\ref{def:minimal}, $\bigcup_{n=0}^{+\infty} F^n(\Omega_+(x))$ has zero area. We will now show that $X=\Lambda_+(x)\smallsetminus \bigcup_{n=0}^{+\infty} F^n(\Omega_+(x))$ has zero area. Take $z\in X$. Lemma~\ref{le:2} yields a sequence $\{n_j\}_{j\in \bN}$ of positive integers such that $g_{F,\Lambda_+(x)}^{n_j}(z)\notin B(\Omega_+(x),\Theta)$. As in the proof of Lemma~\ref{le:3}, after passing to a subsequence if necessary we can assume that $g_{F,\Lambda_+(x)}^{n_j}(z)\in B(y,\epsilon/2)$ for some $y\in\Lambda_+(x)\smallsetminus B(\Omega_+(x),\Theta)$, $g_{F,\Lambda_+(x)}^{n_j}(z)\to y$ and $\bigl| Dg_{F,\Lambda_+(x)}^{n_j}(z) \bigr|\to +\infty$ as $j\to +\infty$.
    By Lemma~\ref{le:2}, for each $n\in \bN$, all branches of $F^n$ are defined on $B(y,2\epsilon).$ For each $j\in \bN$, let $f_{n_j}$ be the branch of $F^{n_j}$ defined on $B(y,2\epsilon) $, defined by the condition that $f_{n_j} \bigl( g_{F,\Lambda_+(x)}^{n_j}(z) \bigr)=z$. Using the K\"obe distortion theorem and 
    \begin{equation*}
        \lim_{j\to +\infty}  Df_{n_j} \bigl( g_{F,\Lambda_+(x)}^{n_j}(z) \bigr) \to 0,
    \end{equation*}
    it follows that $\lim_{j\to +\infty}Df_{n_j}(y)= 0$. By the same theorem, for each $y'\in B(y,\epsilon)$, 
    \begin{equation*}
    	\bigl|Df_{n_j}(y) / Df_{n_j}(y') \bigr|\leq k(1/2).
    \end{equation*}
    Since $\Lambda_+(x)$ is closed and has empty interior, we can find some ball $B(y',\epsilon')\subseteq B(y,\epsilon)\smallsetminus \Lambda_+(x)$. As $B(y',\epsilon')\subseteq U_{z_0}\smallsetminus \Lambda_+(x)$ for some $z_0\in \Lambda_+(x)$ (see Remark~\ref{re:1} and the proof of Lemma~\ref{le:1}), $f_{n_j}(B(y',\epsilon))\cap \Lambda_+(x)=\emptyset$. Hence, $\{f_{n_j}(B(y,\epsilon)\}_{j\in \bN}$ is a sequence of shrinking ovals, i.e., their diameter tends to 0, of bounded shape around $z$ that contain definite gaps $f_{n_j}(B(y',\epsilon'))$ in $\Lambda_+(x)$. Hence, $\Lambda_+(x)$ is porous (see \cite[Section~19.18]{Mishasbook}) at $z$, and so $z$ is not a Lebesgue density point. It follows that $X$ has zero area.
\end{proof}

\begin{lemma}\label{le:atoms}
Suppose that $F$ is relatively hyperbolic on $\Lambda_+(x)$ and that $\Lambda_+(x)$ is minimal. Suppose that $\mu$ is a $\delta$-conformal measure, with $\delta>0$. Then
all atoms are contained in $\bigcup_{n=0}^{+\infty} F^n(\Omega_+(x))$. If for each $\omega\in \Omega_+(x)$, each branch that does not map $\omega$ into $\Omega_+(x)$ has a critical point at $\omega$, then all atoms of $\mu$ are contained in $\Omega_+(x)$.
\end{lemma}

\begin{proof}
In the proof of Lemma~\ref{le:3}, we showed that if  $z\in X$ there exists a sequence of natural numbers $\{n_j\}_{j\in \bN} $ such that 
\begin{equation*} 
\lim_{j\to +\infty } \bigl| Dg_{F,\Lambda_+(x)}^{n_j}(z) \bigr|= +\infty.
\end{equation*}
For each such $z$, suppose it has mass $\gamma>0$. We can find an $n_j$ and a branch $f_{n_j}$ mapping $g_{F,\Lambda_+(x)}^{n_j}(z)$ to $z$ with the property that $\bigl|Df_{n_j} \bigl( g_{F,\Lambda_+(x)}^{n_j}(z) \bigr) \bigr|<\gamma$ so that $g_{F,\Lambda_+(x)}^{n_j}(z)$ has mass strictly greater than $1$, contradicting that the total mass equals $1$. It follows that each atom is contained in $\bigcup_{n=0}^{+\infty} F^n(\Omega_+(x)).$
If each branch of $F$ that does not map $\omega\in \Omega_+(x)$ into $\Omega_+(x)$ has a critical point at $\omega$, then each atom is contained in $\Omega_+(x)$ by conformality of the measure $\mu$.
\end{proof}

\begin{lemma}\label{le:openmeasures} 
   	Suppose that $F$ is relatively hyperbolic on $\Lambda_+(x)$ and that $\Lambda_+(x)$ is minimal. If $\delta>0$ and $\mu$ is a non-open $\delta$-conformal measure, then the support of $\mu$ is contained in $\Omega_+(x)$.
\end{lemma}
   
\begin{proof}
   This statement follows from the conformality of $\mu$, that $\mathcal C_F\cap \Lambda_+(x)\subseteq \Omega_+(x)$ is finite, and that for each $z\in \Lambda_+(x)$, 
   $\bigcup_{n=0}^{+\infty} F^{n}(z)$
   is dense in $\Lambda_+(x)$. Indeed, suppose that the support of $\mu$ is not contained in $\Omega_+(x)$. Then one can find a point $z\in \Lambda_+(x)$ that does not belong to $\overline{\bigcup_{n=0}^{+\infty} F^{-n}(\mathcal C_F\cap \Lambda_+(x))}\subseteq \Omega_+(x)$, such that each open neighborhood of $z$ has positive measure. Then, using that $\bigcup_{n=0}^{+\infty} F^{n}(z)$ is dense in $\Lambda_+(x)$, one can deduce using conformality of $\mu$ that any open subset of $\Lambda_+(x)$ has positive measure.
   %Thus, by Lemma~\ref{le:atoms}, it follows that $\mu \bigr(\bigcup_{n=0}^{+\infty} F^n(\Omega_+(x))\bigr) =1$. Now, for each point $z\in \Lambda_+(x)$ of positive $\mu$-mass, there exists a point $y\in \mathcal C_F\cap \Lambda_+(x)$ and an integer $n\geq 0$ such that that $g_{F,\Lambda_+(x)}^n(y)=z$, for otherwise each point in $\bigcup_{n=0}^{+\infty} F^n(z)$ would have positive mass, contradicting that $\mu$ is not open. But $\mathcal C_F\cap \Lambda_+(x)\subseteq \Omega_+(x),$ so both $y$ and $z$ belong to $ \Omega_+(x)$. 
\end{proof}
\begin{rem}
There are correspondences that admit conformal measures with all their mass contained in $\Omega_+(x)$. For instance, for any Bullett--Penrose correspondence corresponding to a parameter in the modular Mandelbrot set, the measure $\delta_0$, i.e., the Dirac measure at 0, is easily seen to be $\delta$-conformal for any $\delta\in (0,+\infty)$, as the branch that does not fix the parabolic fixed point $0$ has a critical point at 0, see \cite[Proposition~5.1]{bullett2017mating}. In this case, the purely atomic measures are not open. This is in stark contrast to the rational setting.
\end{rem}

\begin{comment}
   First, by conformality of $\mu$, the set
   \begin{equation*}
   	\bigcup_{n=1}^{+\infty} F^n(\mathcal C_F)
   \end{equation*}
   has empty measure. 
   Let $U\subset \Lambda_+(x)$ be an open subset that has $\mu$-measure zero, and assume that there is an open subset $V$ disjoint from
   \begin{equation*}
   	\bigcup_{n=1}^{+\infty} F^n(\mathcal C_F)
   \end{equation*}
   such that $\mu(W)>0$ for each open subset $W\subseteq V$. Take any open subset $W\subset V$. Since $\Lambda_+(x)$ is minimal, 
   \begin{equation*}
   	\bigcup_{n=0}^{+\infty} F^{n}(w)\cap U\neq \emptyset, 
   \end{equation*}
    so by conformality of $\mu$, $\mu(U)>0$. Hence, no such set $V$ exists. This implies that $\mu$ is purely atomic. Then, by Lemma~\ref{le:atoms}, the support of . We have that for each $z\in \Lambda_+(x), $  the set $\bigcup_{n=0}^{+\infty} F^n(z)$ is dense in $\Lambda_+(x)$. Thus, by conformality of $\mu$ and the fact that $\mu$ is non-open, if $z$ has non-zero $\mu$-measure, then 
   \begin{equation*}
   	\mu\Bigl(\mathcal C_F\cap \bigcup_{n=0}^{+\infty} F^n(z) \Bigr)\neq 0.
   \end{equation*}
    Since $\mathcal C_F$ is finite and $F$ is invariantly inverse-like on $\Lambda_+(x)$, it follows that the support of $\mu$ is finite and intersects $\mathcal C_F$.
\end{comment}

We now have the following lemma.

\begin{lemma}\label{le:11} 
Suppose that $F$ is relatively hyperbolic on $\Lambda_+(x)$ and that $\Lambda_+(x)$ is minimal. Then any two open $\delta$-conformal measures $\mu_1$ and $\mu_2$, with $\delta>0$, are equivalent on $X$. Moreover, their Radon--Nikodym derivative $\phi\coloneqq \frac{\,\mathrm{d}\mu_1}{\,\mathrm{d}\mu_2}$ is bounded and satisfies $\phi(f(z))=\phi(z)$ for $\mu_1$-almost every $z\in X$ and each locally defined branch $f$ of $F$. If for each $\omega\in \Omega_+(x)$, each branch that does not map $\omega$ into $\Omega_+(x)$ has a critical point at $\omega$, then the two measures are equivalent on $\Lambda_+ (x) \smallsetminus \Omega_+(x)$. In this case, $\phi$ is bounded and satisfies $\phi(f(z))=\phi(z)$ for $\mu_1$-almost every $z\in \Lambda_+(x)\smallsetminus \Omega_+(x)$ and each locally defined branch $f$ of $F$.
\end{lemma}

\begin{proof}
The proof for the statement that the two measures are equivalent on $\Lambda_+(x)\smallsetminus\bigcup_{n=0}^{+\infty} F^n(\Omega_+(x))=X$, is carried out verbatim as the proof of \cite[Lemma~11]{urbanskidenker1}, so we omit it here.

\begin{comment}
     but let us include it here for completeness. Take a measurable set $E\subseteq X$ and $\epsilon>0$. For each $z\in E$, let $\{r_j(z)\}_{j\in \bN}$ be the sequence Lemma~\ref{le:3} provides. Since the measures $\mu_i$ are measures on $\setRS$, they are regular. Since $r_j(z)\to 0$ as $j\to +\infty$, it follows that for each $z\in E$ there exists $r(z)$ of the form $r_j(z)$ such that

\begin{equation*}
	\mu_2\Bigl(\bigcup_{z\in E}B(z,r(z)))\smallsetminus E\Bigr)<\epsilon.
\end{equation*}
\begin{sloppypar}
\noindent
    Using Besicovic's theorem we can find a countable subcover $\{B(z_i,r(z_i))\}_{i\in \bN}$ of $\{B(z,r(z)\}_{z\in E}$ which has multiplicity bounded by $C$ (not depending on the cover). Hence, 
\end{sloppypar}

\begin{align*}
\mu_1(E)&\leq \sum_{i=1}^{+\infty} \mu_1(B(z_i,r(z_i)))
\\&\leq B(\mu_1)\sum_{i=1}^{+\infty} r(z_i)^{\delta}\\&\leq B(\mu_1)B(\mu_2)\sum_{i=1}^{+\infty} \mu_2(B(z_i,r(z_i)))\\&\leq B(\mu_1)B(\mu_2)C\mu_2\Bigl(\bigcup_{i=1}^{+\infty} (B(z_i,r(z_i)))\Bigr)\\&\leq B(\mu_1)B(\mu_2)C (\epsilon+\mu_2(E)).\end{align*}
Hence, sending $\epsilon\to 0$  we obtain  $$\mu_1(E)\leq B(\mu_1)B(\mu_2)C \mu_2(E).$$ Reversing the roles of $\mu_1$ and $\mu_2$ we get

$$\mu_2(E)\leq B(\mu_1)B(\mu_2)C \mu_1(E).$$ Hence, the two measures are equivalent on $X$  and their Radon--Nikodym derivative is bounded by $$B(\mu_1)B(\mu_2)C .$$ 
\end{comment}

For the statement about the invariance of the Radon--Nikodym derivative, consider the following. For $i=1,2$, by conformality of the measures $\mu_i$, we have $\frac{ d f^*\mu_i }{\,\mathrm{d}\mu_i}=|Df|^{\delta}$ $\mu_i$-a.e.\ locally. Take $z\in X$. For sufficiently small $\gamma>0$, all branches of $F$ are well-defined and injective on $B(z,\gamma)$ and the ball $B(z,\gamma)$ does not contain points in $\mathcal{CV}_{F^{-1}}$. Then
\begin{equation*}
\mu_1(f(B(z,\gamma)))
=\int_{f(B(z,\gamma))}\!\phi \,\mathrm{d}\mu_2
=\int_{B(z,\gamma)}\!\phi\circ f \, \mathrm{d}(f^*\mu_2)
=\int_{B(z,\gamma)}\!|Df|^{\delta} \phi\circ f   \,\mathrm{d}\mu_2.
\end{equation*}
On the other hand,
\begin{equation*}
\mu_1(f(B(z,\gamma)))=\int_{B(z,\gamma)} \!|Df|^{\delta} \,\mathrm{d}\mu_1 =\int_{B(z,\gamma)} \!|Df|^{\delta}  \phi  \,\mathrm{d}\mu_2.
\end{equation*}
Sending $\gamma\to 0$ it follows that $|Df(z)|^{\delta}  \phi(z) =|Df(z)|^{\delta} \phi\circ f (z)$ for $\mu_2$-almost every $z \in X$, i.e., $  \phi = \phi\circ f (z) $ for $\mu_1$-almost every $z \in X$, since $|Df(z)|^{\delta}\neq 0$ for $z\in X$ and the measures $\mu_1$ and $\mu_2$ are equivalent on $X$. 
If for each $\omega\in \Omega_+(x)$, the branches that do not map $\omega$ into $\Omega_+(x)$ have a critical point at $\omega$, then by conformality, $\bigcup_{n=1}^{+\infty}F^{n}(\Omega_+(x))\smallsetminus \Omega_+(x)$ has zero measure, so in this case $\mu_1$ and $\mu_2$ are equivalent on $\Lambda_+ (x) \smallsetminus \Omega_+ (x)$. This also implies the last sentence of the statement of the lemma.
\end{proof}

Together with Lemma~\ref{le:atoms}, the following result is verified the same way as \cite[Theorem~13~(ii)]{urbanskidenker1}, and we refer the reader to that paper for its proof.
 
\begin{lemma}  Suppose that $F$ is relatively hyperbolic on $\Lambda_+(x)$, that $\Lambda_+(x)$ is minimal, and that there exists an open $\delta_1$-conformal measure $\mu_1$. Then for each $\delta_2>\delta_1$, there exists no  $\delta_2$-conformal measure $\mu_2$ that is not purely atomic, and the atoms of each $\delta_2$-conformal measure are contained in $\bigcup_{n=0}^{+\infty} F^n(\Omega_+(x))$.
\end{lemma}

\begin{comment}
    
\begin{proof}
Suppose that there exists an open $\delta_2$-conformal measure $\mu_2$ and take $\epsilon>0$.
For each $z\in X$ find $r(z)$ of the form $r_j(z)$ such that $r(z)<\epsilon^{1/ (\delta_2-\delta_1)}$. Using Besicovic's covering theorem we can find a countable subcover $\{B(z_i,r(z_i))\}_{i\in \bN}$ of $\{B(z,r(z)\}_{x\in X}$ which has multiplicity bounded by $C$. Hence, Lemma~\ref{le:3} yields
\begin{align*}
    \mu_2(X)&\leq \sum_{i=1}^{+\infty} \mu_2(B(z_i,r(z_i)))\\&\leq B(\mu_2)\sum_{i=1}^{+\infty} r(z_i)^{\delta_2}
\\&\leq B(\mu_2)\sum_{i=1}^{+\infty} r(z_i)^{\delta_2-\delta_1}B(\mu_1)\mu_1((B(z_i,r(z_i))))\\&\leq \epsilon B(\mu_1)B(\mu_2)\sum_{i=1}^{+\infty} \mu_1(B(z_i,r(z_i)))\\&\leq \epsilon C B(\mu_1)B(\mu_2)\mu_1\Bigl(\bigcup_{i=1}^{+\infty} (B(z_i,r(z_i)))\Bigr)\\&\leq \epsilon C B(\mu_1)B(\mu_2).\end{align*}
Sending $\epsilon\to 0$ yields $\mu_2(X)=0$.
\end{proof}
\end{comment}

One can show, using the same arguments as in the proof of Lemma~\ref{le:openmeasures}, that each non-purely atomic conformal measure is open and, using Lemma~\ref{le:11}, deduce the following corollary.

\begin{cor}
Suppose that $F$ is relatively hyperbolic on $\Lambda_+(x)$ and that $\Lambda_+(x)$ is minimal. Then there exists at most one real number $\delta$ for which there can exist $\delta$-conformal measures that are not purely atomic. All such measures are equivalent on $X$.
\end{cor}

Using condition~(\ref{it:countable}) in Definition~\ref{def:minimal}, one can now show the following proposition by following the proof of \cite[Theorem~14]{urbanskidenker1}, and we omit the proof.

\begin{prop}\label{prop:HD}
	Suppose that $F$ is relatively hyperbolic on $\Lambda_+(x)$ and that $\Lambda_+(x)$ is minimal.
If $\mu$ is a non-atomic $\delta$-conformal measure for some $\delta\geq0$, and $H_\delta$ is the $\delta$-dimensional Hausdorff measure on $\Lambda_+(x)$, then $H_\delta$ is absolutely continuous with respect to $\mu$ with Radon--Nikodym  derivative $\frac{\mathrm{d}H_\delta}{\mathrm{d}\mu}$ bounded from above. As a consequence, $\operatorname{HD}(\Lambda_+(x)) \leq \delta$ and there exists no non-atomic $t$-conformal measure for any $t\in [0, \operatorname{HD}(\Lambda_+(x)) ).$
\end{prop}

\begin{comment}
\begin{proof}
    Take a Borel set $F\subseteq \Lambda_+(x)$ and let $E=X\cap F$. Note that $H_\delta (E)=H_\delta (F)$. Take $\gamma,\epsilon>0$ such that $$m\Bigr(\bigcup_{z\in E}B(z,r(z))\smallsetminus E\Bigr)<\epsilon$$ and so that for each $z\in E$, $r(z)<\gamma$ is of the form $r_j(z)$ (from Lemma~\ref{le:3}) for some $j\in \bN$. Find a countable subcover $\{B(z_i,r(z_i)\}_{i\in \bN}$ with multiplicity bounded by $C$.
Then \begin{align*}
    \sum_{i=1}^{+\infty} r(z_i)^{\delta}&\leq B(\mu)\sum_{i=1}^{+\infty} m(B(z_i,r(z_i)))\\
    &\leq B(\mu)Cm\Bigl(\bigcup_{i=1}^{+\infty} (B(z_i,r(z_i)))\Bigr)
\\
&\leq B(\mu) C (\epsilon+m(E)).
\end{align*}

Sending $\epsilon\to 0$ and then $\gamma\to 0$, using the definition of the Hausdorff measure yields 
\begin{equation*}
	H_\delta (F)=H_\delta (E)\leq B(\mu) C m(E) \leq B(\mu) C m(F).
\end{equation*}
This concludes the proof of the proposition. \end{proof}
\end{comment}

Applying the ideas of \cite[Theorem~8.8]{urbanskidenker6}, we can now conclude Proposition~\ref{prop:HDmain} stated at the beginning of this section.

\begin{proof}[Proof of Proposition~\ref{prop:HDmain}]
   We can suppose that $\mu$ is non-atomic. Indeed, we can decompose $\mu$ into $\mu_{0}+\mu_{\operatorname{atom}}$, where $\mu_0$ is non-atomic and $\mu_{\operatorname{atom}}$ is purely atomic. Consider $\mu_1=\frac{\mu_0}{|\mu_0|}$, where $|\mu_0|$ denotes the total mass of $\mu_0$. Then $\mu_1$ is a non-atomic $\delta$-conformal measure. So let us assume that $\mu$ is non-atomic and suppose that $\delta\geq 2$.
   By Proposition~\ref{prop:area}, $\Lambda_+(x)$ has zero area and hence zero $\delta$-dimensional Hausdorff measure. Using Lemma~\ref{le:3}, for each $z\in X$, there exists a number $B\geq 1$ and a sequence $\{r_j(z)\}_{j\in \bN}$ of real numbers with $r_j(z)\to 0$ as $j\to +\infty$ such that  
\begin{equation*}
	B^{-1}\leq r_j^{-\delta} \mu(B(z,r_j)) \leq B.
\end{equation*}
Let $\lambda_2$ denote the $2$-dimensional Lebesgue measure and let $\gamma>0$ be given. For each $z\in X$, find $r(z)\leq 1$ of the form $r_j(z)$ such that 
\begin{equation*}
	\lambda_2\Bigl(\bigcup_{z\in X}B(z,r(z))\Bigr)<\gamma.
\end{equation*}
Using the Besicovitch covering theorem, we now find a countable subcover 
\begin{equation*}
	\bigcup_{i=1}^{+\infty} B(z_i,r(z_i))
\end{equation*}
of multiplicity less than or equal to some positive integer $C$.
Then 
\begin{align*}
    \mu(X)
    &\leq \sum_{i=1}^{+\infty} \mu(B(z_i,r(z_i)))
    \leq B\sum_{i=1}^{+\infty} r(z_i)^{\delta}
    \leq B\sum_{i=1}^{+\infty} r(z_i)^2\\
    &=\frac{B}{\pi}\sum_{i=1}^{+\infty} \lambda_2 (B(z_i,r(z_i)))
    \leq \frac{BC}{\pi}\lambda_2 \biggl(\bigcup_{i=1}^{+\infty} B(z_i,r(z_i))\biggr)
    <\frac{BC\gamma}{\pi}.
\end{align*}

Sending $\gamma\to0$ shows that $\mu(X)=0$ and since $\mu$ is non-atomic, $\mu(\Lambda_+(x))=0$. However, $\mu$ is a probability measure and this contradiction implies that $\delta<2. $ Now, Proposition~\ref{prop:HD} gives $\operatorname{HD}(\Lambda_+(x))\leq \delta.$ 
\end{proof}

We are ready to conclude one of the main results of the text.

\begin{proof}[Proof of Theorem~\ref{thm:main}]
	This theorem follows immediately from Propositions~\ref{prop:fund}, \ref{le:nonatomic}, \ref{prop:HDmain}, and Lemma~\ref{le:atoms}.
\end{proof}

\section{Sufficient conditions for \texorpdfstring{$\delta_{\operatorname{crit}(x)}\geq 1$}{critical exponent greater than 1}}\label{sec:suffcond}

In this section, we study the critical exponent $\deltacrit(x)$ and find conditions that bound it from above and below. As a consequence, for any relatively hyperbolic correspondence with limit set $\Lambda_+(x)$ such that $\Lambda_+(x)$ is minimal, we find conditions that imply that there exists a non-atomic conformal measure. 

Suppose that $F$ is invariantly inverse-like on $S$ and that $g_{F,S}$ possesses an attracting periodic orbit $z_1,\,\dots,\, z_n$ of period $n$ in the interior of $S$. For each $z_j$, there is a topological disk $D_j\subseteq F(S)$ containing $z_j$ in its interior, such that $(g_{F,S}^n)^k(z)\to z_j$ as $k\to +\infty$ for each $z\in D_j$. The set $\mathcal A_j\subseteq F(S)$ is the maximal open connected set containing $D_j$ such that $g_{F,S}^k$ is defined on $\mathcal A_j$ for all $k\geq 0$ and for each $z\in \mathcal A_j$, $(g_{F,S}^n)^k(z)\to z_j$ as $k\to +\infty$. The \emph{immediate basin of attraction} of the periodic orbit $z_1,\,\dots,\, z_n$  is $\mathcal A\coloneqq \bigcup_{j=1}^n \mathcal A_j .$

\begin{lemma}\label{le:delta}
Let $F\:z\to w$ be an invariantly inverse-like correspondence on a closed set $S\subsetneq \setRS$. Suppose that the map $g_{F,S}\:F(S)\to S$ has an attracting periodic orbit in the interior of $F(S)$, and let $\mathcal A$ denote the immediate basin of attraction of this attracting periodic point. Then for each $x\in \mathcal{A}\smallsetminus \mathcal{PC}_{F^{-1}}$, it holds that $1\leq \deltacrit(x)\leq2$.
\end{lemma}

\begin{proof}
    Let  an attracting periodic point $z\in \operatorname{int}(S)$ of $g_{F,S}$  of period $k$ and $x\in \mathcal{A}\smallsetminus \mathcal{PC}_{g_{F,S}}$ be given. Take a smooth, closed curve $C\subseteq \mathcal A$ passing through $x$, disjoint from $\mathcal{PC}_{F^{-1}}$ such that one of the connected components $\mathfrak C$ of $\setRS\smallsetminus C$ is such that $\bigl\{g_{F,S}^{nk}(\mathfrak C)\bigr\}_{n\in \bN}$ is a shrinking sequence of topological disks with $z$ in their interiors. This is possible by the dynamics of a holomorphic germ in the basin of an attracting fixed point. Then all branches of $F^n$ are defined in neighborhoods of all points in $C$. Using the compactness of $C$, that $\mathcal A$ is open, and the K\"obe distortion theorem, it follows that for each $s>0$, $P_s(y)$ converges for some $y\in C$ if and only if $P_s(y)$ converges for all $y$ in a neighborhood $U_C$ of $C$. Next, by assumption on $x$, we can find an open neighborhood $U\subseteq U_C$ of $x$ contained in $\mathcal A$, disjoint from $\bigcup_{n=1}^{+\infty} g_{F,S}^n(U)$ and $\mathcal{PC}_{F^{-1}}$. Hence, as $F$ is invariantly inverse-like, all branches of $F^n$ are well-defined and injective on $U$ for each $n\geq 0$, and the image of $U$ under two different branches of $F^n$ and $F^m$ for each pair of integers $n$ and $m$ are disjoint. As the spherical area of $\setRS$ is finite, it follows that 
    \begin{equation*}
        \int_U \!P_2(y)\,\mathrm{d}\lambda_2<+\infty,
    \end{equation*} 
    where, again, $\lambda_2$ denotes the $2$-dimensional Lebesgue measure.
    This implies that
    \begin{equation}\label{eq:conv2}
        P_2(x)<+\infty,
    \end{equation}
    as $U\subseteq U_C$, see also \cite[Proposition~4.3]{mcmullen}.
    
    Next, for each $n\geq 0$, $F^{nk}(\mathfrak C)$ contains a disk with center $z$ and does not intersect $S^c$. Moreover, by continuity of $F^{-nk}$ for each $n\geq 0$, $ \partial F^{nk}(\mathfrak C)\subseteq F^{nk}(C)$. Thus, there exists $c\in (0,+\infty)$ such that
    \begin{equation*}
    	\lambda_1 \bigl( F^{nk}(C) \bigr) \geq c>0,
    \end{equation*}
    for each $n\geq 0$. This implies that
    \begin{equation*}
    	\int_C\! P_1(y) \, \mathrm{d}\lambda_1=+\infty,
    \end{equation*}
    where $\lambda_1$ denotes the 1-dimensional Lebesgue measure.
   
    As $C$ has finite length and for each $s\in (0,+\infty)$, $P_s(y)$ converges for some $y\in C$ if and only if $P_s(y)$ converges for all $y\in C$, it follows that 
    \begin{equation}\label{eq:div1}
        P_1(x)=+\infty.
    \end{equation}
    By definition of the critical exponent, \eqref{eq:conv2} and \eqref{eq:div1} imply that $1\leq \deltacrit(x)\leq 2.$ 
 \end{proof}

    %, which implies that $\deltacrit(x)\leq 2$. suppose that $\deltacrit(x)<1$. Then $P_1(x)<C<+\infty$ for all $x\in C$ by K\"obe distortion theorem, and so $\int_C P_1(x)d\lambda<+\infty$.

We can now show the final result of this paper.

\begin{proof}[Proof of Theorem~\ref{thm:existencemmeasure}]
    This theorem is a direct consequence of combining Propositions~\ref{prop:fund}, \ref{le:nonatomic}, \ref{prop:HDmain}, Lemmas~\ref{le:atoms}, and \ref{le:delta}. 
\end{proof}

\bibliographystyle{amsalpha}
\typeout{}
\bibliography{./theBibliography}

\end{document}